\documentclass[a4paper,oneside,10pt]{article}

\usepackage{graphicx} 
\usepackage{amsmath}
\usepackage{amssymb}
\usepackage{mathtools}
\usepackage{amsthm}
\usepackage[margin=3cm]{geometry}
\usepackage{authblk}

\title{3D Virtual Element Method for Advection-Diffusion-Reaction Problems with Variable Coefficients on Locally Quasi-Uniform Polytopes}
\author[1]{M. Trezzi\thanks{\texttt{manuel.trezzi@unimib.it}}}
\affil[1]{Dipartimento di Matematica e Applicazioni, 
Università degli Studi di Milano-Bicocca, 
Via Roberto Cozzi 55 - 20125 Milano, Italy}

\def\cip{{\rm cip}}

\def\PiNablaE{{\Pi^{\nabla, E}_{k}}}
\def\PiNablaF{{\Pi^{\nabla, F}_{k}}}
\def\PiZeroE{{\Pi^{0, E}_{k}}}
\def\PiZeroF{{\Pi^{0, F}_{k}}}
\def\PiZero{{\Pi^{0}_{k}}}

\def\PiVecMinusOneE{{\mathbf{\Pi}^{0, E}_{k-1}}}

\newcommand{\beps}{\boldsymbol{\varepsilon}}
\newcommand{\bb}{\boldsymbol{\beta}}
\newcommand{\xx}{\mathbf{x}}
\newcommand{\nn}{\mathbf{n}}

\newcommand{\reg}{s}

\newtheorem{theorem}{Theorem}[section]

\newtheorem{proposition}[theorem]{Proposition}
\newtheorem{corollary}[theorem]{Corollary}
\newtheorem{lemma}[theorem]{Lemma}
\newtheorem{remark}[theorem]{Remark}

\newcommand{\R}{\mathbb{R}}

\newcommand{\jump}[1]{\left[\!\left[#1\right]\!\right]}

\def\bskew{{b^{\rm skew}}}
\def\bskewEh{{b^{{\rm skew},E}_h}}

\begin{document}

\maketitle

\begin{abstract}
In this paper, we propose and analyze a Continuous Interior Penalty (CIP) stabilized Virtual Element Method (VEM) for three-dimensional advection-diffusion-reaction equations on general polyhedral meshes. While CIP-VEM schemes have been recently explored in a two-dimensional setting, their analysis heavily relies on global mesh quasi-uniformity and constant physical parameters. We overcome these limitations by introducing a novel three-dimensional variant of the Oswald-type quasi-interpolant. This allows us to establish robust, uniform error estimates in the hyperbolic limit under a realistic local quasi-uniformity assumption and variable coefficients. Finally, we provide a comprehensive set of three-dimensional numerical experiments to validate the theoretical convergence rates and demonstrate the absence of non-physical oscillations.
\end{abstract}

\section{Introduction}
Over the past decade, the Virtual Element Method (VEM) has emerged as a powerful tool for the numerical solution of partial differential equations (PDEs) on general polygonal/polytopal meshes \cite{volley, hitchhikers}. 
By extending the foundational concepts of traditional Finite Element Methods (FEM), VEM bypasses element-geometry constraints, allowing for elements with an arbitrary number of edges and faces, non-convex shapes, and hanging nodes. 
This geometric versatility is a cornerstone of VEM's success, driving its application across a wide spectrum of engineering problems, as highlighted in recent literature \cite{Wriggers2024, sema-simai}.

The objective of this work is the development of a robust polyhedral discretization for three-dimensional advection-diffusion-reaction equations in the hyperbolic limit.
This type of problem poses numerical challenges in the advection-dominated regime. 
In particular, standard Galerkin approximations typically fail, resulting in numerical solutions with severe spurious oscillations across element boundaries.
To overcome this issue, the FEM community has established several stabilization frameworks, including Streamline Upwind Petrov-Galerkin (SUPG) \cite{brooks1982}, Local Projection Stabilization (LPS) \cite{matthies2007unified}, upwind Discontinuous Galerkin (DG) formulations \cite{BrezziDG}, and Continuous Interior Penalty (CIP) methods \cite{burman:2004, douglas}. Some of these techniques has been extend to a VEM framework in \cite{berrone:2016,BBM-2018,BDLV:2021,vem08,VEM-CIP-NC, Trezzi2026}.

Among the aforementioned stabilization techniques, the CIP approach stands out as a ``minimal stabilization'' paradigm, as originally observed in \cite{burman:2005}. In fact, it has been shown that the distance between the continuous convective derivative and the discrete space scales precisely with the magnitude of the gradient jumps across element interfaces.

The CIP technique was recently extended to the VEM framework within a two-dimensional setting in \cite{VEM-CIP-Conf, VEM-CIP-NC}. 
Although these formulations yield robust estimates in the hyperbolic limit, their numerical analysis strictly relies on the assumption of global mesh quasi-uniformity and constant physical parameters. The present work aims to overcome these major restrictions. 
By introducing a novel three-dimensional variant of the Oswald-type quasi-interpolant, we are able to relax the global constraint to a realistic local quasi-uniformity assumption, while simultaneously handling variable coefficients.
The Oswald interpolant is necessary to gain control over the advective term within the energy norm of the problem. Alternatively, one could avoid employing the Oswald operator, but only at the cost of sacrificing the control of the advective derivative in the norm. Such a choice results in a weaker norm that fails to provide robust control over transport instabilities. Furthermore, due to the presence of the Oswald interpolant, boundary conditions must be implemented in a weak sense, leveraging a Nitsche-VEM approach \cite{Bertoluzza2022, PhD-Thesis}.
This advancement is crucial for practical applications: since transport problems frequently develop sharp boundary or internal layers, a local framework allows for highly localized mesh refinement strategies, optimizing computational efforts where they are needed most.

Furthermore, another objective of this study is to provide a comprehensive set of three-dimensional numerical validations. 

The paper is organized as follows. Section \ref{sec:VEM} introduces the continuous model and outlines the construction of the discrete polyhedral spaces. Section \ref{sec:theory} is dedicated to the comprehensive stability and error analysis. Finally, Section \ref{sec:num} validates the theoretical convergence rates through a series of three-dimensional numerical experiments.
Throughout the paper, we adopt standard notations for Sobolev norms and semi-norms.

\section{Problem Setting and Virtual Element Discretization}\label{sec:VEM}

\subsection{The model problem}
Let $\Omega \subset \R^3$ be a polyhedral domain, with boundary $\Gamma$.
The steady advection-diffusion-reaction problem reads as:
\begin{equation}\label{eq:problema-continuo}
\left\{
\begin{aligned} 
&\mbox{find a function $u:\Omega \to \R$ such that:}\\
&- \nabla \cdot ( \beps(\xx) \nabla u) + \bb(\xx) \cdot \nabla u + \sigma(\xx) u 
= f \qquad &\text{in }\Omega , \\
&u = 0 \qquad &\text{on }\Gamma ,
\end{aligned}
\right.
\end{equation}
where $\beps \in [L^\infty(\Omega)]^{3 \times 3}$ is the diffusion tensor, assumed to be symmetric, bounded, and uniformly positive definite, i.e., there exist constants $0 < \beps_0 \leq \beps_1$ such that 
\begin{equation}\label{eq:ellitticità}
\beps_0 |\boldsymbol{\xi}|^2 \leq \boldsymbol{\xi}^T\beps(\xx) \boldsymbol{\xi} \leq \beps_1 |\boldsymbol{\xi}|^2
\end{equation}
for all $\boldsymbol{\xi} \in \R^3$ and almost every $\xx \in \Omega$. The vector field $\bb \in [W^{1, \infty}(\Omega)]^3$ represents the advection velocity, which we assume to be divergence-free, i.e., $\nabla \cdot \bb = 0$ in $\Omega$. The term $\sigma \in L^\infty(\Omega)$ is the reaction coefficient, assumed to be strictly positive, $\sigma(\xx) \geq \sigma_0 > 0$ for almost every $\xx \in \Omega$. Finally, $f \in L^2(\Omega)$ is the external source term. 
The boundary $\Gamma$ is partitioned into two disjoint subsets
\begin{equation*}
\Gamma_{\text{in}} \coloneqq \{ \xx \in \Gamma \, | \, \bb(\xx) \cdot \nn < 0 \}
\quad \text{and} \quad 
\Gamma_{\text{out}} \coloneqq \{ \xx \in \Gamma \, | \, \bb(\xx) \cdot \nn \geq 0 \} \, ,
\end{equation*}
where $\nn$ denotes the outward unit normal vector. 
The variational formulation of \eqref{eq:problema-continuo} reads as
\begin{equation}\label{eq:problem-c}
\left\{
\begin{aligned} 
&\text{find $u \in V \coloneqq H^1_0(\Omega)$ such that: } \\
& a(u,v) + \bskew(u,v) +  c(u,v)  = F(v) \qquad \forall v \in V \, .
\end{aligned}
\right.
\end{equation}
The bilinear forms are defined as
\begin{equation}
\label{eq:a-c}
a(u,  v) \coloneqq \int_{\Omega} (\beps(\xx) \nabla u) \cdot \nabla v \, {\rm d}\Omega
\qquad  \forall u, v \in V \, ,
\end{equation}
\begin{equation}
\label{eq:b-c}
\bskew (u,v) \coloneqq \dfrac{1}{2}\bigl(b(u,  v) - b(v, u)\bigr) \quad\text{with}\quad b(u,  v) \coloneqq \int_{\Omega} (\bb(\xx) \cdot \nabla u) \, v \, {\rm d}\Omega 
\qquad \forall u, v \in V,
\end{equation}
\begin{equation}
\label{eq:c-c}
c(u,  v) \coloneqq \int_{\Omega}  \sigma(\xx) \, u  v \, {\rm d}\Omega
\qquad \forall u, v \in V.
\end{equation}
Note that, due to the homogeneous Dirichlet boundary conditions and the divergence-free assumption on the advection field, $b(u,v)$ is skew-symmetric, hence $b(u,v) = \bskew(u,v)$. 
At the discrete level, we will consider polynomial projections of the virtual functions. This breaks the exact skew-symmetry, which is then restored by explicitly adopting the skew-symmetric formulation.
Similarly, the right-hand side is defined as
\[
F(v)
=
\int_\Omega f \, v \,  {\rm d} \Omega \, .
\]
Standard numerical approximations of problem \eqref{eq:problem-c} notoriously suffer from non-physical oscillations in advection dominated regimes, i.e., whenever the diffusion tensor $\beps$ is relatively small compared to the advective field $\bb$.
The present work adopts the CIP approach, originally formulated in \cite{burman:2004}, to overcome these difficulties.
This approach penalizes the jumps of the gradients across the element boundaries to reduce the number of oscillations.

\subsection{Mesh assumptions}

We consider a family of meshes $\{\Omega_h\}_h$, each representing a conforming partition of the domain $\Omega$ into polyhedra. We denote by $E$ a generic element of the mesh, by $F$ a generic face, and by $e$ a generic edge. 
For any geometric entity $D \in \{E, F, e\}$, we denote its diameter by $h_D$, its measure by $|D|$, and its centroid by $\xx_D$. 
We make the following standard VEM assumptions:
\medskip

\textbf{Mesh assumptions (A1):} there exists a positive constant $\rho$, independent of the mesh size $h$, such that:
\begin{itemize}
    \item each polyhedron $E$ and each face $F$ are star-shaped with respect to a ball of radius greater than or equal to $\rho h_E$ and $\rho h_F$, respectively;
    \item the inequality $h_f \geq \rho h_E$ holds for every element $E$ and every face $F$ of $E$;
    \item the mesh is locally quasi-uniform, i.e., there exists a constant $C \geq 1$ such that $h_E \leq C \, h_{K}$ for any pair of adjacent elements $E$ and $K$.
\end{itemize}
\begin{remark}
Note that the first two assumptions already ensure local quasi-uniformity. We include the third condition for the sake of clarity and to emphasize the departure from the strategy adopted in \cite{VEM-CIP-Conf}.
\end{remark}
Furthermore, we assume that for every polyhedron $ E \in \Omega_h$, the diffusion coefficient  satisfies $\beps \in [W^{1,\infty}(E)]^{3 \times 3}$.

\medskip

For any geometric entity $D \in \{E, F\}$, we define the set of scaled monomials of degree up to $n$ as:
\begin{equation}\label{eq:scaled-monomials}
\mathcal{M}_n(D) \coloneqq \left\{ \left( \frac{\xx - \xx_D}{h_D} \right)^{\boldsymbol{\alpha}} : |\boldsymbol{\alpha}| \leq n \right\} ,
\end{equation}
where $\boldsymbol{\alpha}$ is a multi-index and $|\boldsymbol{\alpha}|$ denotes its total degree.
For any element or face $D$, we introduce the following local polynomial projection operators:
\begin{itemize}
    \item the \textbf{$\boldsymbol{L^2}$-projection} $\Pi_n^{0, D} \colon L^2(D) \to \mathbb P_n(D)$, defined for any $v \in L^2(D)$ by the orthogonality condition:
    \begin{equation*}
        \int_D (v - \Pi_n^{0, D} v) q_n \, \mathrm{d}D = 0 \qquad \text{for all } q_n \in \mathbb P_n(D)\,.
    \end{equation*}
    This operator naturally extends to vector-valued $\boldsymbol{\Pi}_n^{0, D} \colon [L^2(D)]^3 \to [\mathbb P_n(D)]^3$;

    \item the \textbf{$\boldsymbol{H^1}$-seminorm projection} $\Pi_n^{\nabla, D} \colon H^1(D) \to \mathbb P_n(D)$, defined for any $v \in H^1(D)$ as the unique solution to the system:
    \begin{equation*}
        \left\{
        \begin{aligned}
            & \int_D \nabla (v - \Pi_n^{\nabla, D} v) \cdot \nabla q_n \, \mathrm{d}D = 0 \qquad \text{for all } q_n \in \mathbb P_n(D)\,, \\
            & \int_{\partial D} (v - \Pi_n^{\nabla, D} v) \, \mathrm{d}s = 0\,.
        \end{aligned}
        \right.
    \end{equation*}
\end{itemize}
The corresponding global projection operators, $\Pi_n^{0} \colon L^2(\Omega) \to \mathbb P_n(\Omega_h)$ and $\Pi_n^{\nabla} \colon H^1(\Omega_h) \to \mathbb P_n(\Omega_h)$, are defined element-wise by setting:
\begin{equation*}
    (\Pi_n^{0} v)|_E \coloneqq \Pi_n^{0, E} (v|_E) \qquad \text{and} \qquad (\Pi_n^{\nabla} v)|_E \coloneqq \Pi_n^{\nabla, E} (v|_E) \qquad \text{for all } E \in \Omega_h\,.
\end{equation*}
Finally, we recall a classical approximation result for polynomials on star-shaped domains (see, e.g., \cite{brenner-scott:book}).

\begin{lemma}[Polynomial approximation]
\label{lm:bramble}
Under the mesh regularity assumptions stated previously, for any element $E \in \Omega_h$ and any sufficiently smooth function $\phi$ defined on $E$, the following approximation bounds hold:
\begin{equation*}
\begin{alignedat}{2}
    \|\phi - \Pi_n^{0, E} \phi\|_{W^{m,p}(E)} &\lesssim h_E^{s-m} |\phi|_{W^{s,p}(E)} \qquad && \text{for } s, m \in \mathbb{N}, \ m \leq s \leq n+1, \ p \in [1, \infty]\,, \\
    \|\phi - \Pi_n^{\nabla, E} \phi\|_{m, E} &\lesssim h_E^{s-m} |\phi|_{s, E} \qquad && \text{for } s, m \in \mathbb{N}, \ m \leq s \leq n+1, \ s \geq 1\,, \\
    \|\nabla \phi - \boldsymbol{\Pi}_n^{0, E} \nabla \phi\|_{m, E} &\lesssim h_E^{s-1-m} |\phi|_{s, E} \qquad && \text{for } s, m \in \mathbb{N}, \ m+1 \leq s \leq n+1, \ s \geq 1\,.
\end{alignedat}
\end{equation*}
\end{lemma}
\subsection{Virtual Element Spaces}

Before introducing the virtual space on a polyhedron $E$, it is necessary to introduce the virtual space on a face $F \subset \partial E$. 
Given a positive integer $k$, we define the face virtual space as:
\begin{equation}\label{eq:vem-space-face}
\begin{aligned}
V_h^k(F) \coloneqq \bigl\{ v_h \in H^1(F) \cap C^0(\partial F) : \, & v_h|_e \in \mathbb{P}_k(e) \quad \forall e \subset \partial F, \\
& \Delta_{F} v_h \in \mathbb{P}_{k}(F), \\
& (v_h - \PiNablaF v_h, q_k)_{0,F} = 0 \quad \forall q_k \in \mathbb{P}_k(F) \setminus \mathbb{P}_{k-2}(F) \bigr\},
\end{aligned}
\end{equation}
where $\Delta_{F}$ denotes the 2D Laplacian operator on the face $F$.
Given a polyhedron $E\in\Omega_h$, we define the 3D virtual element space of order $k \in \mathbb{N}$ \cite{Ahmad-Alsaedi-Brezzi-Marini-Russo:2013, BeiradaVeiga-Brezzi-Marini-Russo:2014} as
\begin{equation}\label{eq:vem-space-element}
\begin{aligned}
V_h^k(E) \coloneqq
\bigl\{
v_h \in H^1(E)\cap C^0(\partial E) :  \, & v_h|_F \in V_h^k(F) \quad \forall F \subset \partial E, \\
&\Delta v_h \in \mathbb{P}_{k}(E), \\
&(v_h - \PiNablaE v_h, p_k)_{0,E} = 0 \quad \forall p_k \in \mathbb{P}_k(E) \setminus \mathbb{P}_{k-2}(E)
\bigr\}.
\end{aligned}
\end{equation}
The degrees of freedom for $V_h^k(E)$ are typically chosen as follows:
\begin{itemize}
    \item the values of $v_h$ at the vertices of the polyhedron $E$;
    \item for $k > 1$, the values of $v_h$ at the $k-1$ inner points of a Gauss-Lobatto quadrature rule for each edge;
    \item for $k > 1$, the moments of $v_h$ on each face $F \subset \partial E$: $\frac{1}{|F|}\int_F v_h \, m_{k-2} \, {\rm d}F \quad \forall m_{k-2} \in \mathcal{M}_{k-2}(F)$;
    \item for $k > 1$, the internal moments on the element $E$: $\frac{1}{|E|}\int_E v_h \, m_{k-2} \, {\rm d}E \quad \forall m_{k-2} \in \mathcal{M}_{k-2}(E)$.
\end{itemize}
This choice of the DoFs guarantees that all the DoFs scale as $\mathcal O(1).$
The global virtual element space is obtained by gluing the local spaces
\begin{equation}\label{eq:global-virtual-space}
V^k_h(\Omega_h) = \{v_h \in H^1(\Omega) \quad \text{s.t.} \quad v_h|_E \in V^k_h(E) \quad \text{for all $E \in \Omega_h$} \} \, .
\end{equation}
\begin{remark}
We highlight that, since the boundary conditions are imposed weakly, the discrete space is not required to be a subspace of $H_0^1(\Omega)$.
\end{remark}
We conclude this section by recalling the standard interpolation estimate for virtual element spaces.
\begin{lemma}[Approximation using virtual element functions]
\label{lm:interpolation}
Under the assumption \textbf{(A1)} for any $v \in H^{s+1}(\Omega_h)$, there exists $v_I \in V_h^k(\Omega_h)$ such that for all $E \in \Omega_h$ it holds
\[
\|v - v_I\|_{0,E} + h_E \|\nabla (v - v_I)\|_{0,E} \lesssim h_E^{s+1} |v|_{s+1,E} \, , 
\]
where $0 < s \le k$.
\end{lemma}

\subsection{Discrete problem}

As is standard in the VEM context, we replace the bilinear forms introduced in \eqref{eq:a-c}--\eqref{eq:c-c} and the right-hand side with their discrete counterparts.
First, we note that each of these forms can be decomposed into local contributions by restricting the integral to a polyhedron $E \in \Omega_h$:
\[
\begin{aligned}
a(u,v) &\eqqcolon \sum_{E \in \Omega_h} a^E(u,v) \, , \qquad
&b(u,v) \eqqcolon \sum_{E \in \Omega_h} b^E(u,v) \, ,  \\
c(u,v) &\eqqcolon \sum_{E \in \Omega_h} c^E(u,v) \, , \qquad
&F(v) \eqqcolon \sum_{E \in \Omega_h} F^E(v) \, .
\end{aligned}
\]
We now introduce the discrete bilinear forms. For the diffusive contribution, we define:
\begin{equation}\label{eq:ah_E}
a_h^E(u_h, v_h)
\coloneqq
\int_E \beps \, \PiVecMinusOneE \nabla u_h \cdot \PiVecMinusOneE \nabla  v_h \, {\rm d }E
+
h_E \, \bar{\beps}_E \,  S^E_h\bigl( (I - \PiZeroE) u_h, (I-\PiZeroE) v_h \bigr) \, ,
\end{equation}
where $\bar{\beps}_E$ is defined as the trace of the diffusion tensor $\beps$ evaluated at the barycenter of the polyhedron $E$. 
As standard in the VEM context, the bilinear form $S_h^E$ satisfies
\begin{equation}\label{eq:stab-prop}
|v_h|^2_{1,E} \lesssim h_E S_h^E(v_h,v_h) \lesssim |v_h|^2_{1,E} \qquad \forall v_h \in \ker(\PiZeroE)\, .
\end{equation}
Consequently, it holds that
\begin{equation}\label{eq:stab-ah}
a^E(v_h, v_h) \lesssim a_h^E(v_h, v_h) \lesssim a^E(v_h, v_h) \qquad \forall v_h \in V_h^k(E) \, .
\end{equation}
Similarly, the reaction bilinear form is discretized as:
\begin{equation}\label{eq:ch_E}c_h^E(u_h, v_h)\coloneqq\int_E \sigma \,  \PiZeroE u_h \, \PiZeroE v_h \, {\rm d }E
+
\bar{\sigma}_E|E| S_h^E\bigl( (I - \PiZeroE) u_h, (I-\PiZeroE) v_h \bigr) \, ,
\end{equation}
where $\bar{\sigma}_E$ is the mean value of $\sigma$ in the element $E$.
Similarly to \eqref{eq:stab-ah}, we have that
\begin{equation}\label{eq:stab-ch}
c^E(v_h, v_h) \lesssim c_h^E(v_h, v_h) \lesssim c^E(v_h, v_h) \qquad \forall v_h \in V_h^k(E) \, .
\end{equation}
\begin{remark}
Different stabilization terms could have been chosen for the various discrete bilinear forms. For the sake of simplicity, we assume that all these terms are identical.
In particular, we adopt the classical \texttt{dofi-dofi} stabilization.
\end{remark}
The convective bilinear form is discretized by 
\begin{equation}\label{eq:bh_E}
b_h^E(u_h, v_h)\coloneqq\int_E \bigl( \bb \cdot \nabla \PiZeroE u_h \bigr) \, \PiZeroE v_h \, {\rm d }E 
+ 
\sum_{F \subset \partial E}
\int_{F} (\bb \cdot \nn^E) (\PiZeroF - \PiZeroE) u_h  \, \PiZeroF v_h \, {\rm d}F
\, .
\end{equation}
As stated in the model problem, we consider the skew symmetric part
\[
\bskewEh(u_h,v_h) \coloneqq \frac{1}{2} \bigl( b_h^E(u_h, v_h) - b_h^E(v_h, u_h)\bigr) \, .
\]
Before defining the stabilization and boundary terms, for any element $E$ adjacent to the domain boundary $\Gamma$, we define the boundary portion $\Gamma_E \coloneqq \partial E \cap \Gamma$. Furthermore, for any internal face $F = \partial E^+ \cap \partial E^-$ shared by two elements $E^+$ and $E^-$, we define the jump of a sufficiently regular vector field $\boldsymbol{\tau}$ across $F$ as $\jump{\boldsymbol{\tau}} \coloneqq \boldsymbol{\tau}|_{E^+} - \boldsymbol{\tau}|_{E^-}$.
In order to stabilize the method in the hyperbolic limit, we introduce the following bilinear form that penalizes the jump of the gradients across the element boundaries.
We define the penalty parameters as:
\begin{equation}
\gamma_F \coloneqq \gamma \, \|\bb\|^2_{[L^\infty(F)]^3} \quad \text{and} \quad \gamma_E \coloneqq \gamma \, \|\bb\|^2_{[L^\infty(E)]^3} \, ,
\end{equation}
where $\gamma > 0$ is a user-defined constant factor. The stabilization form reads:
\[
\begin{aligned}
J^E_h(u_h, v_h)
\coloneqq
\frac{1}{2}
\sum_{F \subset \partial E \setminus \Gamma_E} \int_F \gamma_F \, h_F^2 \,&\jump{\nabla \PiZero u_h} \cdot \jump{\nabla \PiZero v_h} {\rm d }F \\ &\qquad+
\gamma_E \, |\partial E| \, S^E_h\bigl( (I - \PiZeroE) u_h, (I-\PiZeroE) v_h \bigr) \, .
\end{aligned}
\]
Finally, the Nitsche bilinear form is defined as
\begin{equation*}\label{eq:NE}
\begin{aligned}
\mathcal{N}_h^E(u_h,v_h) 
&\coloneqq
\sum_{F \subset \Gamma_E} \Bigg(
- \int_F (\beps \PiVecMinusOneE \nabla v_h \cdot \nn^E) \, \PiZeroF u_h \, {\rm d}F 
- \int_F (\beps \PiVecMinusOneE \nabla u_h \cdot \nn^E) \, \PiZeroF v_h \, {\rm d}F \\
&\qquad\qquad\quad 
+ \frac{\|\beps\|_{[L^\infty(F)]^{3 \times 3}}}{\delta h_E} \int_F \PiZeroF u_h \PiZeroF v_h \, {\rm d}F   
+ \frac{1}{2} \int_F \vert \bb \cdot \nn^E \vert \, \PiZeroF u_h \PiZeroF v_h \, {\rm d}F \Bigg) \, .
\end{aligned}
\end{equation*}
The right-hand side is obtained by projecting the test functions into the polynomial space
\[
F_h^E(v_h) = \int_E f \, \PiZeroE v_h \, {\rm d}E \, .
\]
We define the cumulative bilinear form $A_\cip^E(\cdot,\cdot)$ as
\[
A_\cip^E(u_h, v_h) = a^E_h(u_h,v_h) + \bskewEh(u_h,v_h) + c^E_h(u_h,v_h) + J^E_h(u_h,v_h) + \mathcal{N}_h^E(u_h,v_h) \, .
\]
The global forms are obtained by summing the local contributions from all the polyhedra
\[
\begin{aligned}
a_h(u_h, v_h) &\coloneqq \sum_{E \in \Omega_h} a_h^E(u_h, v_h) \, , \qquad
&b_h(u_h, v_h) &\coloneqq \sum_{E \in \Omega_h} b_h^E(u_h, v_h) \, , \\
b_h^{\rm{skew}}(u_h,v_h) &\coloneqq \sum_{E \in \Omega_h} \bskewEh(u_h, v_h) \, , \qquad
&c_h(u_h, v_h) &\coloneqq \sum_{E \in \Omega_h} c_h^E(u_h, v_h) \, , \\
J_h(u_h,v_h) &\coloneqq \sum_{E \in \Omega_h} J_h^E(u_h, v_h) \, , \qquad
&\mathcal{N}_h(u_h, v_h) &\coloneqq \sum_{E \in \Omega_h} \mathcal{N}_h^E(u_h, v_h) \, , \\
F_h(v_h) &\coloneqq \sum_{E \in \Omega_h} F_h^E(v_h) \, , \qquad
&A_\cip(u_h, v_h) &\coloneqq \sum_{E \in \Omega_h} A_\cip^E(u_h, v_h) \, .
\end{aligned}
\]
The discrete virtual element problem reads as
\begin{equation}\label{eq:problema-discreto}
\left\{
\begin{aligned}
&\text{find $u_h \in V_h^k(\Omega_h)$ such that:} \\
&A_\cip(u_h, v_h) = F_h(v_h) \quad \forall v_h \in V_h^k(\Omega_h) \, .
\end{aligned}
\right .
\end{equation}
As is standard in the VEM framework, the use of polynomial projections in the definition of the bilinear forms implies that the continuous solution of problem \eqref{eq:problem-c} does not generally satisfy the discrete problem \eqref{eq:problema-discreto}. Nevertheless, assuming sufficient regularity for the exact solution, specifically $u \in H^2(\Omega) \cap H_0^1(\Omega)$, the following consistency identity holds:
\begin{equation}
\label{eq:consistency-1}
\tilde{A}_\cip(u , \, v_h) = F(v_h) \qquad \forall v_h \in V^k_h(\Omega_h) \, .
\end{equation}
The global bilinear form $\tilde{A}_\cip(\cdot,\cdot)$ is decomposed into local contributions $\tilde{A}^E_\cip(\cdot,\cdot)$, defined as:
\begin{equation}
\label{eq:AE-c}
\tilde{A}^E_\cip (u, v_h) 
\coloneqq 
a^E(u, v_h) 
+ b^{\rm{skew},E}(u, v_h) 
+ c^E(u, v_h) 
+ \tilde{\mathcal{N}}_h^E (u, v_h) \, . 
\end{equation}
The modified Nitsche form $\tilde{\mathcal{N}}_h^E(\cdot, \cdot)$ is given by:
\begin{equation*}
\label{eq:tNE}
\begin{aligned}
\tilde{\mathcal{N}}_h^E(u,v_h) 
&\coloneqq
\sum_{F \subset \Gamma_E} \Bigg(
- \int_F (\beps \PiVecMinusOneE \nabla v_h \cdot \nn^E) \, u \, {\rm d}F 
- \int_F (\beps \nabla u \cdot \nn^E) \, v_h \, {\rm d}F \\
&\qquad\qquad\quad 
+ \frac{\|\beps\|_{[L^\infty(F)]^{3 \times 3}}}{\delta h_E} \int_F u \, \PiZeroF v_h \, {\rm d}F   
+ \frac{1}{2} \int_F \vert \bb \cdot \nn^E \vert \, u \, \PiZeroF v_h \, {\rm d}F \Bigg) \, .
\end{aligned}
\end{equation*}

\section{Theoretical analysis}\label{sec:theory}

\subsection{Preliminary results}

We begin by recalling from \cite{VEM-Estimates-2D, VEM-Estimates-3D} the following norm equivalence for virtual functions in two and three dimensions.

\begin{proposition}\label{prp:norm equivalence}
Under assumption \textbf{(A1)}, for any element $E \in \Omega_h$, there exist two positive constants $C_1$ and $C_2$, independent of the mesh size $h_E$, such that
\[
C_1 h_E^3 \sum_{\nu \in \mathcal N_E} \mathrm{dof}^E_\nu(v_h)^2
\leq
\| v_h \|^2_{0,E} 
\leq
C_2 h_E^3 \sum_{\nu \in \mathcal N_E} \mathrm{dof}^E_\nu(v_h)^2 \qquad \forall v_h \in V_h^k(E)\, ,
\]
where $\mathcal{N}_E$ denotes the set of local DoFs associated with the element $E$, and $\mathrm{dof}^E_\nu(v_h)$ is the value of the DoF $\nu$ evaluated for $v_h$. 
Similarly, for any face $F$, it holds that
\[
C_1 h_F^2 \sum_{\nu \in \mathcal N_F} \mathrm{dof}^F_\nu(v_h)^2
\leq
\| v_h \|^2_{0,F} 
\leq
C_2 h_F^2 \sum_{\nu \in \mathcal N_F} \mathrm{dof}^F_\nu(v_h)^2 \qquad \forall v_h \in V_h^k(F) \, ,
\]
where $\mathcal{N}_F$ denotes the set of local DoFs associated with the face $F$, and $\mathrm{dof}^F_\nu(v_h)$ is the value of the DoF $\nu$ on the face $F$. 
We remark that $\mathcal{N}_F$ includes not only the moments internal to the face but also the DoFs associated with the edges and vertices belonging to $\partial F$.
\end{proposition}
\noindent Furthermore, we will make use of the following inverse inequality for virtual functions \cite{VEM-Estimates-3D}.
\begin{proposition}\label{prp:vem-inverse}
Under assumption \textbf{(A1)}, for any $E \in \Omega_h$, it holds that
\[
| v_h |_{1,E} \lesssim h_E^{-1} \| v_h \|_{0,E} \qquad \forall v_h \in V_h^k(E),
\]
where the implied constant is independent of the mesh size $h_E$.
\end{proposition}
\noindent As a direct consequence of this inverse estimate, we obtain the following trace inequality.

\begin{corollary}\label{crl:trace}
Under assumption \textbf{(A1)}, for any $E \in \Omega_h$, it holds that
\[
\| v_h \|_{0,\partial E} \lesssim h_E^{-\frac{1}{2}} \| v_h \|_{0,E} \qquad \forall v_h \in V_h^k(E),
\]
where the implied constant is independent of the mesh size $h_E$.
\end{corollary}

\begin{proof}
The result follows by combining a standard trace inequality with the inverse estimate provided in Proposition \ref{prp:vem-inverse}.
\end{proof}

Following \cite{VEM-CIP-Conf, PhD-Thesis, burman:2004}, we introduce the 3D Oswald interpolant, also known as averaging operator, for virtual element functions. 
This operator constructs a continuous approximation $\pi(v_h)$ from a piecewise defined, potentially discontinuous function $v_h$ by averaging its DoFs across neighboring elements. Specifically, for any global degree of freedom $\nu$ associated with the space $V_h^k(\Omega_h)$, the corresponding DoF of the interpolant is defined as:
\[
\mathrm{dof}_\nu(\pi(v_h))
\coloneqq
\frac{1}{|\omega_\nu|}
\sum_{E \in \omega_\nu}
\mathrm{dof}^E_\nu(v_h) \, ,
\]
where $\omega_\nu = \{ E \in \Omega_h : \nu \text{ is a local DoF of } E \}$ is the patch of elements sharing $\nu$, and $| \omega_\nu |$ denotes its cardinality.
\begin{proposition}\label{prop:local_oswald_estimate}
Let $\nu$ be a global DoF associated with the space $V_h^k(\Omega_h)$. 
We define the locally averaged mesh size at $\nu$ as:
\[
h_\nu \coloneqq \frac{1}{|\omega_\nu|} \sum_{E \in \omega_\nu} h_E \, .
\] 
Let $p_h$ be a piecewise polynomial function defined on the mesh $\Omega_h$ of degree at most $k$. We introduce a discontinuous virtual element function $\tilde{w}_h$ by prescribing its local DoFs on each element $E \in \Omega_h$ as follows:
\[
\mathrm{dof}^E_\nu(\tilde{w}_h) \coloneqq h_\nu \, \mathrm{dof}^E_{\nu}(p_h) \, .
\]
Finally, we define the continuous approximation $w_h \coloneqq \pi(\tilde{w}_h) \in V_h^k(\Omega_h)$. 
Under assumption \textbf{(A1)}, it holds that for any element $E \in \Omega_h$:
\[
\| \tilde{w}_h - w_h \|_{0,E}^2 \le C \, h_E \sum_{F \in \mathcal{F}(\omega_E)} h_F^2 \, \| \jump{p_h} \|_{0,F}^2 \, ,
\]
where $C>0$ is a constant independent of the mesh size. Here, $\omega_E$ denotes the set of elements in $\Omega_h$ sharing at least one vertex with $E$, and $\mathcal{F}(\omega_E)$ is the set of all internal faces $F$ such that $F = \partial E_1 \cap \partial E_2$ with $E_1, E_2 \in \omega_E$.
\end{proposition}

\begin{proof}
Let us focus on a specific element $E \in \Omega_h$. Since the difference $(\tilde{w}_h - w_h)$ belongs to the local VEM space $V^k_h(E)$, we can apply Proposition \ref{prp:norm equivalence} and obtain
\begin{equation}\label{eq:norm_equiv_3d}
\| \tilde{w}_h - w_h \|_{0,E}^2 \lesssim h_E^3 \sum_{\nu \in \mathcal{N}_E} \left| \mathrm{dof}^E_\nu(\tilde{w}_h) - \mathrm{dof}^E_\nu(w_h) \right|^2 \, .
\end{equation}
Exploiting the definitions of $w_h$ and $\tilde{w}_h$, we observe that DoF $\nu $ of $w_h$ evaluates to:
\[
\mathrm{dof}_\nu(w_h) 
=
\mathrm{dof}^E_\nu(w_h) 
= 
\frac{1}{|\omega_\nu|} \sum_{K \in \omega_\nu} \mathrm{dof}^K_\nu(\tilde{w}_h)
= 
h_\nu \left( \frac{1}{|\omega_\nu|} \sum_{K \in \omega_\nu} \mathrm{dof}^K_\nu(p_h) \right) 
\eqqcolon 
h_\nu \, \mathrm{dof}_\nu(q_h) \, ,
\]
where $q_h \coloneqq \pi(p_h)$ is the standard Oswald interpolant of the piecewise polynomial $p_h$ \cite{VEM-CIP-Conf}. 
Consequently, we can rewrite the difference inside the sum in \eqref{eq:norm_equiv_3d} as:
\begin{equation}\label{eq:dof_diff_3d}
\left| \mathrm{dof}^E_\nu(\tilde{w}_h) - \mathrm{dof}^E_\nu(w_h) \right| = h_\nu \left| \mathrm{dof}^E_\nu(p_h) - \mathrm{dof}^E_\nu(q_h) \right| \, .
\end{equation}
Moreover, if $\nu$ is a volume moment DoF associated with $E$ (which by definition is internal and not shared with other elements), we have that $\omega_\nu = \{E\}$ and the difference \eqref{eq:dof_diff_3d} is equal to zero.
Hence the difference \eqref{eq:norm_equiv_3d} only sees the DoFs attached to a face $F \subset \partial E$. It holds that
\begin{equation} \label{eq:norm_equiv_3d-2} 
h_E^3 \sum_{\nu \in \mathcal{N}_E} \left| \mathrm{dof}^E_\nu(\tilde{w}_h) - \mathrm{dof}^E_\nu(w_h) \right|^2
\leq
h_E^3
\sum_{F \subset \partial E} \sum_{\nu \in \mathcal{N}_F} h^2_\nu \left| \mathrm{dof}^F_\nu(p_h|_E) - \mathrm{dof}^F_\nu(q_h) \right|^2 \, .
\end{equation}
The second inequality holds because, under the mesh regularity assumption \textbf{(A1)}, each edge and vertex degree of freedom is shared by a finite and uniformly bounded number of faces.
Hence, following the standard arguments used for the 2D case \cite{VEM-CIP-Conf}, we deduce that:
\[
\left| \mathrm{dof}^F_\nu(p_h|_E) - \mathrm{dof}_\nu(q_h) \right|^2 \lesssim \sum_{F \in \mathcal{F}(\omega_\nu)} \left| \mathrm{dof}^F_\nu(\jump{p_h}_F) \right|^2 \, .
\]
Substituting this estimate into \eqref{eq:norm_equiv_3d-2}, swapping the sums, we can bound the sum in \eqref{eq:norm_equiv_3d} as follows:
\[
\| \tilde{w}_h - w_h \|_{0,E}^2 \lesssim h_E^3   \sum_{F \in \mathcal{F}(\omega_E)} \sum_{\nu \in \mathcal N_F} h_\nu^2 \left| \mathrm{dof}^F_\nu(\jump{p_h}_F) \right|^2 \, .
\]
Exploiting the mesh local quasi-uniformity, and Proposition \ref{prp:norm equivalence}, we obtain:
\[
\| \tilde{w}_h - w_h \|_{0,E}^2 
\lesssim 
h_E^3 \sum_{F \in \mathcal{F}(\omega_E)} \| \jump{p_h} \|_{0,F}^2 
\lesssim
h_E \sum_{F \in \mathcal{F}(\omega_E)} h_F^2 \| \jump{p_h} \|_{0,F}^2 \, .
\]
\end{proof}

\subsection{Inf-sup condition}

Given a polyhedron $E \in \Omega_h$, we define the local CIP norm as
\begin{equation}\label{eq:normaCIP-E}
\begin{aligned}
| v_h |^2_{\cip, E}
&\coloneqq
\| \beps^{\frac12} \nabla v_h \|^2_{0,E} 
+
h_E \, \| \bb \cdot \nabla \PiZeroE v_h \|^2_{0,E} +
\| \sigma^{\frac{1}{2}} v_h \|^2_{0,E} \\
&\quad 
+ 
\sum_{F \subset \Gamma_E} 
\left(
\| \xi_{\beps, F} \PiZeroF v_h \|^2_{0,F} 
+
\| \xi_{\bb, F} \PiZeroF v_h \|^2_{0,F}
\right) 
+ 
J_h^E(v_h,v_h) \, ,
\end{aligned}
\end{equation}
and the global norm as
\begin{equation}\label{eq:normaCIP}
\| v_h \|_{\cip} \coloneqq 
\left( 
\sum_{E \in \Omega_h} \| v_h \|^2_{\cip,E}
\right)^\frac12 \, .
\end{equation}
The quantities appearing in the boundary norms are defined as follows:
\begin{equation}
\label{eq:xi-def}
\xi_{\beps, F} \coloneqq \left( \dfrac{\| \beps \|_{[L^\infty(F)]^{3 \times 3}}}{\delta \, h_E} \right)^{\frac{1}{2}} \, , 
\qquad
\xi_{\bb, F} \coloneqq \left( \dfrac{1}{2} | \bb \cdot \nn^E | \right)^{\frac{1}{2}} .
\end{equation}
The proof of the inf-sup condition is structured into two main steps. First, we establish control over the symmetric terms of the bilinear form by testing it with identical entries.
Second, we demonstrate how to control the convective term appearing in the norm by employing a specifically constructed test function.
\begin{proposition}\label{prp:symm} 
Under assumption \textbf{(A1)}, let $v_h \in V_h^k(\Omega_h)$, then the following estimate holds
\[
\begin{aligned}
A_\cip (v_h,v_h)
&\gtrsim \sum_{E \in \Omega_h} \Big( \,
\| \beps^{\frac12} \nabla v_h \|^2_{0,E}
+
\| \sigma^{\frac{1}{2}} v_h \|^2_{0,E} + J_h^E(v_h,v_h) \\
&\qquad \qquad + \sum_{F \subset \Gamma_E}
\bigl
( \, \| \xi_{\beps, F} \PiZeroF v_h \|^2_{0,F}
+
\| \xi_{\bb, F} \PiZeroF v_h \|^2_{0,F}
\, \bigr) \Big) \, .
\end{aligned}
\]
\end{proposition}

\begin{proof}
Since by construction the bilinear form $b_{\rm skew}(\cdot, \cdot)$ is skew-symmetric, exploiting the coercivity of the discrete forms \eqref{eq:stab-ah} and \eqref{eq:stab-ch}, there exists a constant $\alpha_* > 0$, depending only on the stabilization term $S_h^E(\cdot,\cdot)$, such that:
\begin{equation}\label{eq:Acipvhvh}
\begin{aligned}
A^E_\cip (v_h, v_h)
&\geq
\alpha_* \left( \| \beps^{\frac12} \nabla v_h \|^2_{0,E} + \| \sigma^{\frac{1}{2}} v_h \|^2_{0,E} \right) +
J_h^E(v_h,v_h) \\
&\quad + \sum_{F \subset \Gamma_E} \left( \| \xi_{\beps, F} \PiZeroF v_h \|^2_{0,F} + \| \xi_{\bb, F} \PiZeroF v_h \|^2_{0,F} \right) \\
&\quad - 2 \sum_{F \subset \Gamma_E} \int_F (\beps \PiVecMinusOneE \nabla v_h \cdot \nn^E) \, \PiZeroF v_h \, \mathrm{d}F \, .
\end{aligned}
\end{equation}
To bound the last term, we apply the Cauchy-Schwarz inequality followed by Young's inequality. For any face $F \subset \Gamma_E$ and an arbitrary face-dependent parameter $\eta_F > 0$, we have:
\begin{equation}\label{eq:young}
2 \int_F (\beps \PiVecMinusOneE \nabla v_h \cdot \nn^E) \, \PiZeroF v_h \, {\rm d}F
\leq
\eta_F \, h_E \, \| \beps \PiVecMinusOneE \nabla v_h \cdot \nn^E \|^2_{0,F} + \frac{1}{\eta_F \, h_E} \, \| \PiZeroF v_h \|^2_{0,F} \, .
\end{equation}
Let us bound the first term on the right-hand side. Extracting the local $L^\infty$-norm of the diffusion tensor over the face $F$, we obtain:
\[
h_E \, \| \beps \PiVecMinusOneE \nabla v_h \cdot \nn^E \|^2_{0,F}
\leq
\|\beps\|^2_{[L^\infty(F)]^{3 \times 3}} \, h_E \, \| \PiVecMinusOneE \nabla v_h \|^2_{0,F} \, .
\]
Applying a polynomial trace inequality, it holds $h_E \| \PiVecMinusOneE \nabla v_h \|^2_{0,F} \leq C_{\rm inv} \| \PiVecMinusOneE \nabla v_h \|^2_{0,E}$, where $C_{\rm inv} > 0$ depends only on the mesh shape parameter $\rho$. 
Using the global ellipticity bound \eqref{eq:ellitticità}, and the $L^2$-stability of the projection operator $\PiVecMinusOneE$, we get:
\begin{equation}\label{eq:beps-estimate-boundary}
\begin{aligned}
h_E \| \beps \PiVecMinusOneE \nabla v_h \cdot \nn^E \|^2_{0,F}
&\leq
C_{\rm inv} \|\beps\|^2_{[L^\infty(F)]^{3 \times 3}} \| \PiVecMinusOneE \nabla v_h \|^2_{0,E} \\
&\leq
C_{\rm inv} \|\beps\|^2_{[L^\infty(F)]^{3 \times 3}} \| \nabla v_h \|^2_{0,E} \\
&\leq
C_{\rm inv} \frac{\|\beps\|^2_{[L^\infty(F)]^{3 \times 3}}}{\beps_0} \| \beps^{\frac12} \nabla v_h \|^2_{0,E} \, .
\end{aligned}
\end{equation}
We now choose the parameter $\eta_F = \frac{2\delta}{\|\beps\|_{[L^\infty(F)]^{3 \times 3}}}$ to explicitly match the penalty term in the formulation. Substituting this into \eqref{eq:young}, one power of the local norm cancels out:
\[
2 \int_F (\beps \PiVecMinusOneE \nabla v_h \cdot \nn^E) \, \PiZeroF v_h \, \mathrm{d}F
\leq
2 \delta C_{\rm inv} \frac{\|\beps\|_{[L^\infty(F)]^{3 \times 3}}}{\beps_0} \| \beps^{\frac12} \nabla v_h \|^2_{0,E} 
+ \frac{1}{2} \frac{\|\beps\|_{[L^\infty(F)]^{3 \times 3}}}{\delta h_E} \| \PiZeroF v_h \|^2_{0,F} \, .
\]
Summing this bound over all faces $F \subset \Gamma_E$ and inserting it back into the expansion of $A^E_\cip (v_h, v_h)$ \eqref{eq:Acipvhvh}, we obtain:
\[
\begin{aligned}
A^E_\cip (v_h, v_h)
&\geq
\left(\alpha_* - 2 \delta C_{\rm inv} \sum_{F \subset \Gamma_E} \frac{\|\beps\|_{[L^\infty(F)]^{3 \times 3}}}{\beps_0} \right) \| \beps^{\frac12} \nabla v_h \|^2_{0,E}
+ \alpha_* \| \sigma^{\frac{1}{2}} v_h \|^2_{0,E}
+ J_h^E(v_h,v_h) \\
&\quad + \sum_{F \subset \Gamma_E} \left( \frac12  \| \xi_{\beps, F} \PiZeroF v_h \|^2_{0,F} + \| \xi_{\bb, F} \PiZeroF v_h \|^2_{0,F}  \right)  \, .
\end{aligned}
\]
Let $N_F^{\rm max}$ be the maximum number of faces of a generic element in $\Omega_h$. To ensure that the coefficient of the volume gradient term remains strictly positive, we choose the penalty parameter $\delta$ sufficiently small such that:
\[
\delta \leq \frac{\alpha_* \, \beps_0}{4 \, C_{\rm inv} N_F^{\rm max} \beps_1} \, .
\]
Under this condition, and recalling the definition of the boundary norm \eqref{eq:xi-def}, we conclude that:
\[
\begin{aligned}
A^E_\cip (v_h, v_h)
&\geq
\min\left(\frac{\alpha_*}{2}, 1\right) \Bigg( \| \beps^{\frac12} \nabla v_h \|^2_{0,E}
+ \| \sigma^{\frac{1}{2}} v_h \|^2_{0,E} + J_h^E(v_h,v_h) \\
&\quad + \sum_{F \subset \Gamma_E} \left( \| \xi_{\beps, F} \PiZeroF v_h \|^2_{0,F} + \| \xi_{\bb, F} \PiZeroF v_h \|^2_{0,F} \right) \Bigg) \, .
\end{aligned}
\]
Summing over all polyhedral elements $E \in \Omega_h$ yields the desired global coercivity estimate.
\end{proof}
Before proving the second part of the inf-sup condition, we prove the following lemma.
\begin{lemma}\label{lm:wh-estimate}
Under assumption $\textbf{(A1)}$, let $\bb_h$ be the piecewise constant approximation of $\bb$, and let $p_h \coloneqq \bb_h \cdot \nabla \PiZero v_h$ be a piecewise polynomial function. 
Given $v_h \in V_h^k(\Omega_h)$, let $w_h \in V_h^k(\Omega_h)$ be the virtual function constructed as in Proposition \ref{prop:local_oswald_estimate} starting from $p_h$. It holds that
\[
\|  w_h \|_{0,E}
\lesssim
h_E \| \bb_h \cdot \nabla \PiZero v_h \|_{0,\mathcal D(E)} \, ,
\]
 where $\mathcal D(E)$ denotes the set of elements adjacent to $E$. 
\end{lemma}
\begin{proof}
Let $\tilde w_h$ be the piecewise function defined as in the statement of Proposition \ref{prop:local_oswald_estimate}.
Applying the triangle inequality, we have that
\[
\| w_h \|^2_{0,E}
\lesssim
\| \tilde w_h \|^2_{0,E}
+
\| w_h - \tilde w_h \|^2_{0,E} \, .
\]
Applying the definition of $\tilde w_h$, Proposition \ref{prp:norm equivalence}, and the local quasi-uniformity of $\Omega_h$, we obtain
\begin{equation}\label{eq:oswald-stab-1}
\| \tilde w_h \|^2_{0,E} \lesssim h^2_E \| \bb_h \cdot \nabla \PiZeroE v_h \|^2_{0,E} \, .
\end{equation}
For the second term, using Proposition \ref{prop:local_oswald_estimate} and trace inequality, we obtain
\begin{equation}\label{eq:oswald-stab-2}
\begin{aligned}
\| \tilde{w}_h - w_h \|_{0,E}^2 
&\lesssim 
h_E \sum_{F \in \mathcal{F}(\omega_E)} h_F^2 \, \| \jump{p_h} \|_{0,F}^2 \\
&\lesssim
h_E \sum_{K \in \mathcal{D}(E)} h_K \, \| \bb_h \cdot \nabla \Pi^{0,K}_k v_h \|_{0,K}^2
\, .
\end{aligned}
\end{equation}
Gathering \eqref{eq:oswald-stab-1} and \eqref{eq:oswald-stab-2}, combined with mesh assumption \textbf{(A1)}, and taking the square root, we obtain the inequality that we want to prove.
\end{proof}
\begin{proposition}\label{prp:skew}
Under assumption \textbf{(A1)}, for every $v_h \in V_h^k(\Omega_h)$, it holds that
\[
A_\cip(v_h, w_h) \gtrsim C^*\sum_{E \in \Omega_h}
h_E \, \| \bb \cdot \nabla \PiZeroE v_h \|^2_{0,E}
-
C_* A_\cip (v_h,v_h) \, ,
\]
where $w_h$ is the function defined in Lemma \ref{lm:wh-estimate}.
\end{proposition}
\begin{proof}
We start again by fixing an element $E \in \Omega_h$. By definition of the bilinear form $A_\cip^E(\cdot, \cdot)$, it holds that
\[
A_\cip^E(v_h, w_h) = a^E_h(v_h,w_h) + \bskewEh(v_h,w_h) + c^E_h(v_h,w_h) + J^E_h(v_h,w_h) + \mathcal{N}_h^E(v_h,w_h) \, .
\]
We proceed by bounding each of these five terms.

\medskip 

\noindent \textit{Estimate of $a^E_h(v_h,w_h)$:}
For the diffusive term, using the continuity \eqref{eq:stab-ah} of the discrete bilinear form $a_h^E(\cdot, \cdot)$, the ellipticity of the diffusion tensor \eqref{eq:ellitticità}, the inverse estimate of Proposition \ref{prp:vem-inverse}, and Lemma \ref{lm:wh-estimate}, we obtain
\begin{equation}\label{eq:a_h estimate}
\begin{aligned}
a_h^E(v_h, w_h)
&\gtrsim - \| \beps^\frac12 \nabla v_h \|_{0,E} \| \beps^\frac12 \nabla w_h \|_{0,E} \\
&\gtrsim - \| \beps^\frac12 \nabla v_h \|_{0,E} \left( \beps_1^\frac{1}{2} h_E^{-1} \| w_h \|_{0,E} \right) \\
&\gtrsim - \| \beps^\frac12 \nabla v_h \|_{0,E} \left( \beps_1^\frac{1}{2} \, \| \bb_h \cdot \nabla \PiZeroE v_h \|_{0,\mathcal{D}(E)} \right) \\
&\gtrsim
- \| \beps^\frac12 \nabla v_h \|_{0,E} \, \left( \left(\frac{\beps_1}{\beps_0}\right)^\frac{1}{2} \, \| \bb \|_{[L^\infty(\mathcal D(E))]^3} \| \beps^\frac{1}{2}  \nabla  v_h \|_{0,\mathcal{D}(E)} \right) \\
&\gtrsim
-  \left(\frac{\beps_1}{\beps_0}\right)^\frac{1}{2} \, \| \bb \|_{[L^\infty(\mathcal D(E))]^3} \| \beps^\frac{1}{2}  \nabla  v_h \|^2_{0,\mathcal{D}(E)}  \, .
\end{aligned}
\end{equation}

\medskip 

\noindent \textit{Estimate of $c^E_h(v_h,w_h)$:}
Similarly, we estimate the reaction term using \eqref{eq:stab-ch} and Lemma \ref{lm:wh-estimate}
\begin{equation}\label{eq:c_h estimate}
\begin{aligned}
c_h^E(v_h, w_h)
&\gtrsim - \| \sigma^{\frac{1}{2}} v_h \|_{0,E} \| \sigma^{\frac{1}{2}} w_h \|_{0,E} \\
&\gtrsim - \| \sigma^{\frac{1}{2}} v_h \|_{0,E} \left( \| \sigma \|^\frac12_{L^\infty(E)} h_E \| \bb_h \cdot \nabla \Pi^0_k v_h \|_{0,\mathcal{D}(E)} \right) \, .
\end{aligned}
\end{equation}

\medskip 

\noindent \textit{Estimate of $J^E_h(v_h,w_h)$:}
The estimate of the jump bilinear form follows by applying Cauchy-Schwarz inequality
\[
J_h^E(v_h,w_h) \gtrsim - (J_h^E(v_h, v_h))^\frac{1}{2} (J_h^E(w_h, w_h))^\frac{1}{2} \, .
\]
The second term in this last inequality is handled by a polynomial trace inequality, the properties of the stabilization term \eqref{eq:stab-prop}, the inverse estimate in Proposition \ref{prp:vem-inverse}, and Lemma \ref{lm:wh-estimate}
\begin{equation}\label{eq:j_h estimate}
\begin{aligned}
J_h^E(w_h, w_h) 
&=
\frac{1}{2} \sum_{F \subset \partial E \setminus \Gamma_E} \int_F \gamma_F \, h_F^2 \, \jump{\nabla \PiZero w_h}^2 {\rm d }F
+
\gamma_E \, |\partial E| \, S^E_h\bigl( (I - \PiZeroE) w_h, (I-\PiZeroE) w_h \bigr) \\
&\lesssim \gamma_E h_E^{-1} \| w_h \|^{2}_{0,\mathcal D(E)} + \gamma_E h_E | w_h |^{2}_{1,E} \\
& \lesssim \gamma_E \, h_E^{-1} \, \| w_h \|^{2}_{0,\mathcal D(E)}  
\lesssim
\, \gamma_E \, h_E \,  \| \bb_h \cdot \nabla \PiZero v_h \|^2_{0,\mathcal D (\mathcal D(E))} \, .
\end{aligned}
\end{equation}
\medskip 

\noindent \textit{Estimate of $\mathcal N^E _h(v_h,w_h)$:}
Concerning the Nitsche bilinear form, by definition, we have that
\[
\begin{aligned}
\mathcal{N}_h^E(v_h,w_h) 
\coloneqq
\sum_{F \subset \Gamma_E} \! \Bigg(
&- \int_F (\beps \PiVecMinusOneE \nabla w_h \cdot \nn^E) \, \PiZeroF v_h \, {\rm d}F 
- \int_F (\beps \PiVecMinusOneE \nabla v_h \cdot \nn^E) \, \PiZeroF w_h \, {\rm d}F \\
&
+ \frac{\|\beps\|_{[L^\infty(F)]^{3 \times 3}}}{\delta h_E} \int_F \PiZeroF v_h \PiZeroF w_h \, {\rm d}F   
+ \frac{1}{2} \int_F \vert \bb \cdot \nn^E \vert \, \PiZeroF v_h \PiZeroF w_h \, {\rm d}F \Bigg) \, .
\end{aligned}
\]
On the first one, recalling the definition \eqref{eq:xi-def}, arguing as in \eqref{eq:beps-estimate-boundary} and in \eqref{eq:a_h estimate}, we obtain
\begin{equation}\label{eq:Nitsche-infsup-1}
\begin{aligned}
- \int_F (\beps \PiVecMinusOneE \nabla w_h \cdot \nn^E) \, \PiZeroF v_h \, {\rm d}F 
&\gtrsim - \Vert \xi_{\beps, F} \PiZeroF v_h \Vert_{0,F} \, h_E^\frac{1}{2} \, \| \beps^{\frac{1}{2}} \PiVecMinusOneE \nabla w_h \cdot \nn^E\|_{0,F} \\
&\gtrsim - \Vert \xi_{\beps, F} \PiZeroF v_h \Vert_{0,F} \, \left(\frac{\beps_1}{\beps_0}\right)^\frac{1}{2} \,  \| \beps^\frac{1}{2}\nabla w_h \|_{0,E} \\ 
&\gtrsim - \Vert \xi_{\beps, F} \PiZeroF v_h \Vert_{0,F} \, \left( \frac{\beps_1}{\beps_0} \, \| \bb \|_{[L^\infty(\mathcal D(E))]^3} \| \beps^\frac{1}{2}  \nabla  v_h \|_{0,\mathcal{D}(E)} \right) \, .
\end{aligned}
\end{equation}
For the second term, arguing as in \eqref{eq:beps-estimate-boundary} and \eqref{eq:a_h estimate}, applying Corollary \ref{crl:trace}, and Lemma \ref{lm:wh-estimate}, we have that
\begin{equation}\label{eq:Nitsche-infsup-2}
\begin{aligned}
- \int_F (\beps \PiVecMinusOneE \nabla v_h \cdot \mathbf{n}^E) \, \Pi^{0,F}_k w_h \, {\rm d}F
&\gtrsim 
- \| \beps \PiVecMinusOneE \nabla v_h \cdot \mathbf{n}^E \|_{0,F} \| \PiZeroF w_h \|_{0,F}\\
&\gtrsim 
- \left( \frac{\beps_1}{\sqrt{\beps_0}} h_E^{-\frac{1}{2}} \| \beps^\frac12 \nabla v_h \|_{0,E} \right) \left( h_E^{-\frac{1}{2}} \| w_h \|_{0,E} \right) \\
&\gtrsim 
- \left(\frac{\beps_1}{\sqrt{\beps_0}} h_E^{-1} \| \beps^\frac12 \nabla v_h \|_{0,E}\right) \, \left( h_E \, \| \bb_h \cdot \nabla \Pi^0_k v_h \|_{0,\mathcal D(E)} \right) \\
&\gtrsim 
- \frac{\beps_1}{\sqrt{\beps_0}} \| \beps^\frac12 \nabla v_h \|_{0,E} \, \| \bb \|_{[L^\infty(\mathcal{D}(E))]^3} \| \nabla \Pi^0_k v_h \|_{0,\mathcal D(E)} \\
&\gtrsim 
- \frac{\beps_1}{\beps_0} \| \bb \|_{[L^\infty(\mathcal{D}(E))]^3}  \, \| \beps^\frac12 \nabla v_h \|^2_{0,\mathcal D(E)} \, .
\end{aligned}
\end{equation}
From the definition of $\xi$ in \eqref{eq:xi-def}, exploting \eqref{eq:beps-estimate-boundary} and \eqref{eq:a_h estimate}, we can write:
\begin{equation}\label{eq:Nitsche-infsup-3}
\begin{aligned}
\frac{\|\beps\|_{[L^\infty(F)]^{3 \times 3}}}{\delta h_E} \int_F \PiZeroF v_h \PiZeroF w_h \, {\rm d}F 
&\gtrsim 
-\Vert \xi_{\beps, F} \PiZeroF v_h \Vert_{0,F} \left( \sqrt{\frac{\beps_1}{h_E}} \Vert w_h \Vert_{0,F} \right) \\
&\gtrsim 
-\Vert \xi_{\beps, F} \PiZeroF v_h \Vert_{0,F} \left( \frac{\sqrt{\beps_1}}{h_E} \Vert w_h \Vert_{0,E} \right) \\
&\gtrsim 
-\Vert \xi_{\beps, F} \PiZeroF v_h \Vert_{0,F} \,\frac{\sqrt{\beps_1}}{h_E} \left( h_E \| \bb_h \cdot \nabla \PiZero v_h \|_{0,\mathcal D(E)} \right) \\
&\gtrsim 
-\Vert \xi_{\beps, F} \PiZeroF v_h \Vert_{0,F} \, \left( \frac{\beps_1}{\beps_0}\right)^{\frac12} \, \| \bb \|_{[L^\infty(\mathcal D (E))]^3} \| \beps^\frac{1}{2} \nabla v_h \|_{0,\mathcal D(E)} \, .
\end{aligned}
\end{equation}
Similarly, the estimate of the convective term in the Nitsche form yields
\begin{equation}\label{eq:Nitsche-infsup-4}
\begin{aligned}
\frac{1}{2} \int_F \vert \bb \cdot \nn^E \vert \, \PiZeroF v_h \PiZeroF w_h \, {\rm d}F 
&\gtrsim
-\Vert \xi_{\bb, F} \PiZeroF v_h \Vert_{0,F} \| \bb_h \|_{[L^\infty(F)]^3}^{\frac12} \,\Vert\PiZeroF w_h \Vert_{0,F} \\
&\gtrsim
-\Vert \xi_{\bb, F} \PiZeroF v_h \Vert_{0,F} \| \bb_h \|_{[L^\infty(F)]^3}^{\frac12} \, h_E^{-\frac12} \, \Vert w_h \Vert_{0,E} \\
&\gtrsim 
-\Vert \xi_{\bb, F} \PiZeroF v_h \Vert_{0,F} \, \left(h_E^{\frac{1}{2}} \, \| \bb \|_{[L^\infty(F)]^3}^{\frac12} \, \| \bb_h \cdot \nabla \PiZero v_h \|_{0,\mathcal D(E)} \right)\, .
\end{aligned}
\end{equation}
Combining estimates \eqref{eq:Nitsche-infsup-1}, \eqref{eq:Nitsche-infsup-2}, \eqref{eq:Nitsche-infsup-3}, and \eqref{eq:Nitsche-infsup-4}, and summing over all boundary faces $F \subset \Gamma_E$, we arrive at the final estimate for the Nitsche bilinear form:
\[
\begin{aligned}
\mathcal{N}_h^E(v_h,w_h) 
\gtrsim 
- \sum_{F \subset \Gamma_E} &\Bigg( \Vert \xi_{\beps, F} \PiZeroF v_h \Vert_{0,F} \, \left( \frac{\beps_1}{\beps_0} \, \| \bb \|_{[L^\infty(\mathcal D(E))]^3} \| \beps^\frac{1}{2}  \nabla  v_h \|_{0,\mathcal{D}(E)} \right)\\
&\qquad + \frac{\beps_1}{\beps_0} \| \bb \|_{[L^\infty(\mathcal{D}(E))]^3}  \, \| \beps^\frac12 \nabla v_h \|^2_{0,\mathcal D(E)} \\
&\qquad +\Vert \xi_{\beps, F} \PiZeroF v_h \Vert_{0,F} \, \left( \left(\frac{\beps_1}{\beps_0}\right)^{\frac12} \, \| \bb \|_{[L^\infty(\mathcal D (E))]^3} \| \beps^\frac{1}{2} \nabla v_h \|_{0,\mathcal D(E)} \right) \\
&\qquad +\Vert \xi_{\bb, F} \PiZeroF v_h \Vert_{0,F} \, \left(h_E^{\frac{1}{2}} \, \| \bb \|_{[L^\infty(F)]^3}^{\frac12} \, \| \bb_h \cdot \nabla \PiZero v_h \|_{0,\mathcal D(E)}\right) \Bigg) \, .
\end{aligned}
\]

\medskip 

\noindent \textit{Estimate of $\bskewEh(v_h,w_h)$:}
The last term to estimate is the convective term. By definition, we have that
\begin{equation} \label{eq:infsup-estimate-bh}
\bskewEh(v_h, w_h) 
=
\frac{1}{2} \bigl( b_h^E(v_h, w_h) - b_h^E(w_h, v_h)\bigr) \, .
\end{equation}
Recalling the definition of $b_h^E(v_h, w_h)$ \eqref{eq:bh_E}, we begin by splitting its volume term as
\[
\begin{aligned}
\bigl ( \bb \cdot \nabla \PiZeroE v_h, \PiZeroE w_h \bigr)_{0,E}
&= 
\bigl ( \bb \cdot \nabla \PiZeroE v_h, w_h \bigr)_{0,E}
+ 
\bigl ( \bb \cdot \nabla \PiZeroE v_h, (\PiZeroE - I) w_h \bigr)_{0,E} \\
&= 
\bigl ( \bb \cdot \nabla \PiZeroE v_h, \tilde w_h \bigr)_{0,E}
+ 
\bigl ( \bb \cdot \nabla \PiZeroE v_h, w_h - \tilde w_h \bigr)_{0,E} \\
&\qquad +\bigl ( \bb \cdot \nabla \PiZeroE v_h, (\PiZeroE - I) w_h \bigr)_{0,E} \\
&= 
\bigl ( \bb_h \cdot \nabla \PiZeroE v_h, \tilde w_h \bigr)_{0,E}
+ 
\bigl ( (\bb - \bb_h) \cdot \nabla \PiZeroE v_h, \tilde w_h \bigr)_{0,E}
\\
&\qquad +\bigl ( \bb \cdot \nabla \PiZeroE v_h, w_h - \tilde w_h \bigr)_{0,E} 
+
\bigl ( \bb \cdot \nabla \PiZeroE v_h, (\PiZeroE - I) w_h \bigr)_{0,E} \\
&\eqqcolon \eta_{\bb,1} + \eta_{\bb,2} + \eta_{\bb,3} + \eta_{\bb,4}\, .
\end{aligned}
\]
For the moment, we do not consider the boundary term in $b^E_h(v_h,w_h)$.
Concerning $\eta_{\bb,1}$, exploiting the link between the degrees of freedom of $\tilde{w}_h$ and $\bb_h \cdot \nabla \PiZeroE v_h$, we can lower bound the inner product with the weighted $L^2$-norm of $p_h$, and applying Lemma \ref{lm:bramble}, we obtain that
\[
\begin{aligned}
\eta_{\bb,1} 
&=
\bigl ( \bb_h \cdot \nabla \PiZeroE v_h, \tilde w_h \bigr)_{0,E} \geq
C_1 \, h_E \, \| \bb_h \cdot \nabla \PiZeroE v_h \|_{0,E}^2 \\
& \geq
C_1 \, h_E \, \| \bb \cdot \nabla \PiZeroE v_h \|_{0,E}^2
- C_1 \, h_E \, \| (\bb - \bb_h) \cdot \nabla \PiZeroE v_h \|_{0,E}^2 \\
& \geq
C_1 \, h_E \, \| \bb \cdot \nabla \PiZeroE v_h \|_{0,E}^2
- C_1 \, h_E \, | \bb |^2_{[W^{1,\infty}(E)]^3}\,{\sigma_0^{-1}} \, \| \sigma^\frac{1}{2} v_h \|_{0,E}^2 \, .
\end{aligned}
\]
To bound $\eta_{\bb,2}$, we use Cauchy-Schwarz inequality, Lemma \ref{lm:bramble} and obtain
\[
\begin{aligned}
\eta_{\bb,2 } 
&\geq 
-\|\bb - \bb_h\|_{[L^\infty(E)]^3} \|\nabla \PiZeroE v_h\|_{0,E} \|\tilde w_h \|_{0,E} \\
&\gtrsim
- | \bb |_{[W^{1,\infty}(E)]^3} \| v_h \|_{0,E} \, h_E \,  \| \bb_h \cdot \nabla \PiZeroE v_h \|_{0,E}  \\
&\gtrsim
- h_E^\frac{1}{2}\,| \bb |_{[W^{1,\infty}(E)]^3}\,{\sigma_0^{-\frac{1}{2}}} \| \sigma^\frac{1}{2} v_h \|_{0,E} \, \left(h_E^\frac{1}{2} \,  \| \bb_h \cdot \nabla \PiZeroE v_h \|_{0,E} \right)  \, .
\end{aligned}
\]
Using Cauchy-Schwarz inequality, Proposition \ref{prop:local_oswald_estimate}, the estimate of $\eta_{\bb,3}$ reads as
\[
\begin{aligned}
\eta_{\bb,3}
&\geq
- \| \bb \cdot \nabla \PiZeroE v_h \|_{0,E} \| w_h - \tilde w_h  \|_{0,E} \\
&\gtrsim
- C_2 \, h_E \| \bb \cdot \nabla \PiZeroE v_h \|^2_{0,E} 
- \frac{1}{C_2 \, h_E}\| w_h - \tilde w_h  \|^2_{0,E} \\
&\gtrsim
- C_2 \, h_E \, \| \bb \cdot \nabla \PiZeroE v_h \|^2_{0,E} 
- \sum_{F \in \mathcal{F}(\omega_E)} h_F^2 \, \| [\bb \cdot \nabla \PiZeroE v_h] \|_{0,F}^2
\end{aligned}
\]
Here $C_2$ is chosen in order to guarantee that $C_1 > C_2$ since we want to keep the positivity of the term $ h_E \, \| \bb \cdot \nabla \PiZeroE v_h \|^2_{0,E} $.
Exploiting the same steps of Equation (3.33) and (3.34) in \cite{VEM-CIP-Conf} for estimating the jumps of $\bb \cdot \nabla \PiZeroE v_h$, we derive 
\[
\eta_{\bb,3}
\gtrsim
- C_2 \, h_E \, \| \bb \cdot \nabla \PiZeroE v_h \|^2_{0,E} 
- C_3
\left(
J_h^{\mathcal D(E)}(v_h,v_h) + h_E \, | \bb |^2_{[W^{1,\infty}(\mathcal D(E))]^3} \, \sigma_0^{-1} \, \| \sigma^\frac12 v_h \|^2_{0,\mathcal{D}(E)}
\right) \, .
\]
Finally, on $\eta_{\bb,4}$ exploiting the orthogonality of $\PiZeroE$, and Lemma \ref{lm:bramble}, we obtain
\[
\begin{aligned}
\eta_{\bb,4} &=
\bigl ( \bb \cdot \nabla \PiZeroE v_h, (\PiZeroE - I) w_h \bigr)_{0,E} =
\bigl ( (\bb - \bb_h) \cdot \nabla \PiZeroE v_h, (\PiZeroE - I) w_h \bigr)_{0,E} \\
&\gtrsim
-| \bb |_{[W^{1,\infty}(E)]^3} \, \sigma_0^{-1} \,\| \sigma ^\frac{1}{2}v_h \|^{2}_{0,\mathcal D(E)} \, .
\end{aligned}
\]
We estimate the second term in \eqref{eq:infsup-estimate-bh}. Integrating by parts, and adding and subtracting the term $\bigl( (\bb \cdot \nn^E) \PiZeroF w_h, \PiZeroF 
v_h \bigr)_{0,F}$, we obtain
\[
\begin{aligned}
- b_h^E(w_h, v_h) 
& = - \bigl( \bb \cdot \nabla \PiZeroE w_h, \PiZeroE v_h \bigr)_{0,E}
-
\sum_{F \subset \partial E}\bigl( (\bb \cdot \nn^E) (\PiZeroF - \PiZeroE) w_h, \PiZeroE v_h \bigr)_{0,F}
\\
& =
\bigl( \bb \cdot \nabla \PiZeroE v_h, \PiZeroE w_h \bigr)_{0,E}
- 
\sum_{F \subset \partial E}
\bigl( (\bb \cdot \nn^E)  \PiZeroF w_h, \PiZeroE v_h \bigr)_{0,F} \\
& =
\bigl( \bb \cdot \nabla \PiZeroE v_h, \PiZeroE w_h \bigr)_{0,E}
- 
\sum_{F \subset \partial E}
\bigl( (\bb \cdot \nn^E)  \PiZeroF w_h, (\PiZeroE - \PiZeroF) v_h \bigr)_{0,F} \\
& \quad - 
\sum_{F \subset \partial E}
\bigl( (\bb \cdot \nn^E) \PiZeroF w_h, \PiZeroF  
v_h \bigr)_{0,F} \, .
\end{aligned}
\]

The first term is estimated exactly as in the previous case. The second term cancels with the boundary term in the definition of $b_h^E(v_h,w_h)$.
The last term need to be estimated only on the boundary faces, since the projection $\PiZeroF$ does not depend on the element $E$.
The estimate of these boundary integral is the same of \eqref{eq:Nitsche-infsup-4}.
Combining the bounds for $\eta_{\bb,1}$, $\eta_{\bb,2}$, $\eta_{\bb,3}$, $\eta_{\bb,4}$, we can collect all the terms to obtain the local estimate for the skew-symmetric convective form:
\begin{equation} \label{eq:bskew-final-local-updated}
\begin{aligned}
\bskewEh(v_h, w_h) 
&\gtrsim (C_1 - C_2) \, h_E \| \bb \cdot \nabla \PiZeroE v_h \|_{0,E}^2 \\
&\quad - C \Bigg( 
h_E \, | \bb |^2_{[W^{1,\infty}(E)]^3}\,{\sigma_0^{-1}} \, \|\sigma^\frac{1}{2} v_h \|_{0,E}^2 \\
&\qquad\quad +  h_E^\frac{1}{2}\,| \bb |_{[W^{1,\infty}(E)]^3}\,{\sigma_0^{-\frac{1}{2}}} \| \sigma^\frac{1}{2} v_h \|_{0,E} \, \left(h_E^\frac{1}{2} \,  \| \bb_h \cdot \nabla \PiZeroE v_h \|_{0,E} \right)  \\
&\qquad\quad +  J_h^{\mathcal D(E)}(v_h,v_h) + h_E \, | \bb |^2_{[W^{1,\infty}(\mathcal D(E))]^3} \, \sigma_0^{-1} \, \| \sigma^\frac12 v_h \|^2_{0,\mathcal{D}(E)}\\
&\qquad\quad+ | \bb |_{[W^{1,\infty}(E)]^3} \, \sigma_0^{-1} \,\| \sigma ^\frac{1}{2}v_h \|^{2}_{0,\mathcal D(E)} \\
&\qquad\quad + \sum_{F \subset \partial E \cap \Gamma} \Vert \xi_{\bb, F} \PiZeroF v_h \Vert_{0,F} \, \left(h_E^{\frac{1}{2}} \, \| \bb \|_{[L^\infty(F)]^3}^{\frac12} \, \| \bb_h \cdot \nabla \PiZero v_h \|_{0,\mathcal D(E)} \right)  \Bigg) \, .
\end{aligned}
\end{equation}
By appropriately choosing the stabilization parameters (as discussed for $\eta_{\bb,3}$), we ensure that $C_1 > C_2$, thus guaranteeing that the first term provides a strictly positive control on the streamline derivative. 

\medskip

We introduce the local energy semi-norm:
\begin{equation}
\begin{aligned}
\mathcal{E}_h^E(v_h, v_h) &\coloneqq \| \beps^\frac12 \nabla v_h \|_{0,\mathcal{D}(E)}^2 + \| \sigma^\frac12 v_h \|_{0,\mathcal{D}(E)}^2 + J_h^{\mathcal{D}(E)}(v_h,v_h) \\
& \qquad + \sum_{F \subset \Gamma_E} \left( \| \xi_{\beps, F} \PiZeroF v_h \|^2_{0,F} + \| \xi_{\bb, F} \PiZeroF v_h \|^2_{0,F} \right)   \, ,
\end{aligned}
\end{equation}
which collects all the symmetric and boundary penalty contributions. 
We now consider all the local bounds derived above. Summing over all the elements $E \in \Omega_h$ and exploiting assumption \textbf{(A1)}, we note that each element and its associated patches are counted at most a uniformly bounded number of times. 
We obtain
\begin{equation} \label{eq:global_sum_master}
\begin{aligned}
A_\cip(v_h, w_h) 
&\geq (C_1 - C_2) \sum_{E \in \Omega_h} h_E \| \bb \cdot \nabla \PiZeroE v_h \|_{0,E}^2 - C_3 \sum_{E \in \Omega_h} \mathcal{E}_h^E(v_h, v_h) \\
&\quad - C_4 \left( \sum_{E \in \Omega_h} \mathcal{E}_h^E(v_h, v_h) \right)^\frac{1}{2} \left( \sum_{E \in \Omega_h} h_E \| \bb_h \cdot \nabla \PiZeroE v_h \|_{0,E}^2 \right)^\frac{1}{2} \, .
\end{aligned}
\end{equation}
To absorb the last term using Young's inequality, we first need to relate the streamline derivative along $\bb_h$ back to the continuous field $\bb$. Adding and subtracting $\bb$, we obtain:
\[
\| \bb_h \cdot \nabla \PiZeroE v_h \|_{0,E}^2 
\lesssim
\| \bb \cdot \nabla \PiZeroE v_h \|_{0,E}^2 +  \| (\bb_h - \bb) \cdot \nabla \PiZeroE v_h \|_{0,E}^2 \, .
\]
The second term on the right-hand side can be bounded using the approximation properties of $\bb_h$, similarly to what was done for $\eta_{\bb,1}$:
\[
h_E \| (\bb_h - \bb) \cdot \nabla \PiZeroE v_h \|_{0,E}^2 \lesssim h_E \, | \bb |^2_{[W^{1,\infty}(E)]^3} \, \sigma_0^{-1} \, \| \sigma^\frac{1}{2} v_h \|_{0,E}^2 \lesssim \mathcal{E}_h^E(v_h, v_h) \, .
\]
Therefore, applying Young's inequality to the mixed term in \eqref{eq:global_sum_master} with a sufficiently small constant $\eta > 0$, and substituting the bound above, we obtain:
\begin{equation}
A_\cip(v_h, w_h) \geq C^* \sum_{E \in \Omega_h} h_E \| \bb \cdot \nabla \PiZeroE v_h \|_{0,E}^2 - C_* \sum_{E \in \Omega_h} \mathcal{E}_h^E(v_h, v_h) \, .
\end{equation}
Since the sum of the local energy terms $\mathcal{E}_h^E(v_h, v_h)$ is exactly controlled by the symmetric part of the bilinear form, we obtain the desired global bound.
\end{proof}
Before concluding this section by proving the inf-sup condition, we have to prove the following Lemma.
\begin{lemma}\label{lm:continuity}
Under assumption \textbf{(A1)}, let $w_h$ be the function defined in Lemma \ref{lm:wh-estimate}. It holds that
\[
\| w_h \|_{\cip,E} \lesssim \| v_h \|_{\cip,\mathcal{D}(E)} \qquad \forall E \in \Omega_h \, .
\]
\end{lemma}

\begin{proof}
We consider the volume terms in the definition of the local energy norm. Exploiting Proposition \ref{prp:vem-inverse} and the $L^2$ bound for $w_h$ provided by Lemma \ref{lm:wh-estimate}, we obtain:
\[
\begin{aligned}
\| \beps^\frac12 \nabla w_h \|^2_{0,E} + \| \sigma^\frac{1}{2} w_h \|^2_{0,E}
&\lesssim
\beps_1 h_E^{-2} \|  w_h \|^2_{0,E} + \| \sigma \|_{L^\infty(E)} \| w_h \|_{0,E}^2 \\
&\lesssim
\left( \beps_1 h_E^{-2} + \| \sigma \|_{L^\infty(E)} \right) h_E^2 \| \bb_h \cdot \nabla \Pi_0^E v_h \|^2_{0,\mathcal{D}(E)} \\
&\lesssim
 \frac{\beps_1}{\beps_0} \| \bb \|^2_{[L^\infty(\mathcal D(E))]^3} \| \beps^\frac{1}{2} \nabla v_h \|^2_{0,\mathcal D(E)} \\
& \qquad +  \| \sigma \|_{L^\infty(E)} h_E^2  \| \bb_h \cdot \nabla \Pi_0^E v_h \|^2_{0,\mathcal{D}(E)} \, .
\end{aligned}
\]
Similarly, for the boundary penalty terms appearing in the Nitsche formulation, we can apply Corollary \ref{crl:trace}, and proceed with the same bound:
\[
\begin{aligned}
\| \xi_{\beps, F} \PiZeroF v_h \|^2_{0,F} + \| \xi_{\bb, F} \PiZeroF v_h \|^2_{0,F} 
&\lesssim 
\left(\frac{\beps_1}{h_E} + \| \bb \|_{[L^\infty(F)]^3} \right) \| w_h\|^2_{0,F} \\
&\lesssim
\left(\frac{\beps_1}{h^2_E} + \frac{\| \bb \|_{[L^\infty(F)]^3}}{h_E} \right) \| w_h\|^2_{0,E} \\
&\lesssim
\frac{\beps_1}{\beps_0} \| \bb \|^2_{[L^\infty(E)]^3} \| \beps^\frac{1}{2} \nabla v_h \|^2_{0,\mathcal{D}(E)}\\
& \qquad +
 \| \bb \|^2_{[L^\infty(E)]^3} h_E \| \bb_h \cdot \nabla \Pi^0_h v_h \|^2_{0,\mathcal D(E)} \, .
\end{aligned}
\]
The jump terms associated with the CIP stabilization applied to $w_h$ are controlled by exploiting \eqref{eq:j_h estimate}. 

\medskip

Consequently, all terms constituting the local CIP norm of $w_h$ are strictly bounded by terms appearing in $\| v_h \|_\cip$.
Summing this local bound over all elements $E \in \Omega_h$ and exploiting the finite overlap of the element patches $\mathcal{D}(E)$ under the mesh regularity assumption \textbf{(A1)}, we recover the global continuity estimate $\| w_h \|_{\text{cip}} \lesssim \| v_h \|_{\text{cip}}$.
\end{proof}

\begin{theorem}\label{th:inf-sup}
Under assumptions \textbf{(A1)}, it holds:
\begin{equation*} \label{eq:infsup}
\| v_h \|_{\cip} 
\lesssim 
\sup_{z_h \in V^k_h(\Omega_h)} \dfrac{A_\cip (v_h, z_h)}{\| z_h \|_{\cip}}
\qquad \text{for all $v_h \in V^k_h(\Omega_h)$.}
\end{equation*}
\end{theorem}

\begin{proof}
Let $v_h \in V^k_h(\Omega_h)$ be given.
We construct the test function as $z_h = w_h + \kappa v_h$, where $w_h$ is the function introduced in Lemma \ref{lm:wh-estimate}. From Proposition \ref{prp:symm} and Proposition \ref{prp:skew}, for a sufficiently large constant $\kappa > 0$, we obtain:
\begin{equation} \label{eq:infsup_adv}
A_\cip(v_h, z_h) = A_\cip(v_h, w_h + \kappa v_h) \gtrsim \| v_h \|_\cip^2 \, .
\end{equation}
Applying Lemma \ref{lm:continuity}, we have $\| z_h \|_{\cip} \lesssim \| v_h \|_{\cip}$, which concludes the proof for this case.
\end{proof}

\subsection{Error Analysis}
\label{sub:error}

We begin the error analysis of the method by establishing an abstract result. This result decomposes the total error into an interpolation component and a series of consistency terms.

\begin{proposition}
\label{prp:abstract}
Let $u \in V$ denote the  solution of problem \eqref{eq:problem-c} and $u_h \in V^k_h(\Omega_h)$ be the discrete solution to \eqref{eq:problema-discreto}. Let $u_I \in V^k_h(\Omega_h)$ represent the interpolant of $u$ as defined in Lemma \ref{lm:interpolation}, and let the interpolation error be denoted by
\[
e_I \coloneqq u - u_I.
\]
Assuming that condition \textbf{(A1)} is satisfied, the following error bound holds:
\begin{equation}
\label{eq:abstract}
\| u - u_h \|_{\cip} 
\lesssim 
\| e_I \|_{\cip}
+ 
\sum_{E \in \Omega_h} \bigl(
\eta_F^E + \eta_a^E + \eta_b^E + \eta_c^E + \eta_N^E + \eta_J^E 
\bigr) \, ,
\end{equation}
where the residuals $\eta_i^E$ are defined as follows:
\[
\begin{aligned}
\eta_F^E &\coloneqq \| F^E - F_h^E \|_{\cip^*} \, , 
& \eta_a^E &\coloneqq \| a^E(u, \cdot) - a_h^E(u_I, \cdot) \|_{\cip^*} \, , \\
\eta_b^E &\coloneqq \|b^{\rm{skew},E}(u, \cdot) - \bskewEh(u_I, \cdot) \|_{\cip^*} \, , 
& \eta_c^E &\coloneqq \| c^E(u, \cdot) -  c_h^E(u_I, \cdot) \|_{\cip^*} \, , \\
\eta_N^E &\coloneqq \| \Tilde{\mathcal{N}}_h^E(u, \cdot) - \mathcal{N}_h^E(u_I, \cdot) \|_{\cip^*} \, , 
& \eta_J^E &\coloneqq \|J_h^E(u_I, \cdot)\|_{\cip^*} \, .
\end{aligned}
\]
Here, $\|\cdot\|_{\cip^*}$ denotes the dual norm of $\|\cdot\|_{\cip}$.
\end{proposition}

\begin{proof}
Let us introduce the auxiliary error term:
\[
e_h \coloneqq u_h - u_I \, .
\]
By invoking the triangle inequality, we can split the global error into the interpolation contribution and the discrete error:
\[
\| u - u_h \|_{\cip} \leq \| u - u_I \|_{\cip} + \| u_I - u_h \|_{\cip} = \| e_I \|_{\cip} + \| e_h \|_{\cip} .
\]
By utilizing the inf-sup condition in Theorem \ref{th:inf-sup} and the consistency property \eqref{eq:consistency-1}, we obtain:
\begin{equation*}
\begin{aligned}
\| e_h \|_{\cip} 
& \lesssim 
\sup_{v_h \in V^k_h(\Omega_h)} \dfrac{A_\cip(u_h - u_I, v_h)} {\| v_h \|_{\cip}} 
=
\sup_{v_h \in V^k_h(\Omega_h)} \dfrac{F_h(v_h) - A_\cip(u_I, v_h)}{\| v_h \|_{\cip}} \\
& =
\sup_{v_h \in V^k_h(\Omega_h)} \dfrac{F_h(v_h) -  {F}(v_h) + \tilde{A}_{\cip} (u, v_h) - A_{\cip}(u_I, v_h)}{\| v_h \|_{\cip}} \\
& =
\sup_{v_h \in V^k_h(\Omega_h)} \dfrac{\sum_{E \in \Omega_h}\bigl( F_h^E(v_h) - {F}^E(v_h) + \tilde{A}_{\cip}^E (u, v_h) - A^E_{\cip}(u_I, v_h) \bigr)} {\| v_h \|_{\cip}} \, .
\end{aligned}
\end{equation*}
The final estimate \eqref{eq:abstract} is reached by applying the definitions of the discrete and continuous operators.
\end{proof}

\medskip\noindent
\textbf{(A2) Regularity assumptions.}
To properly bound the terms in Proposition \ref{prp:abstract}, we assume that the solution $u$, the force $f$, and the advective field $\bb$ satisfy:
\begin{gather*}
u \in H^2(\Omega) \cap H^{\reg+1}(\Omega_h) \, , \qquad
f \in H^{\reg+1}(\Omega_h) \, , \qquad
\beps \in [W^{\reg,\infty}(\Omega_h)]^{3\times 3} \, , \\[1ex]
\bb \in [W^{\reg+1,\infty}(\Omega_h)]^3 \, , \qquad
\sigma \in W^{\reg +1}(\Omega_h) \, ,
\end{gather*}
for some $1< \reg \leq k$. Moreover, we introduce the following parameters
\[
\lambda_E \coloneqq \max \{ h_E^2 \, \| \sigma \|_{L^\infty(E)}, \, \beps_1 \} \, ,
\qquad
\nu_E \coloneqq \min \{ \sigma^{-1}_0 \, h_E^{-2}, \, \beps_0^{-1}\} \, .
\]
Clearly, in the context of constant parameters, it holds $\nu_E = \lambda_E^{-1}$.

\begin{remark}
In the following proofs, we frequently encounter terms involving the difference between a function and its projected interpolant, typically of the form $\| u - \PiZeroE u_I \|_{0,E}$ (or analogous expressions). Such terms are handled by applying a triangle inequality to separate the projection and interpolation components:
\[
\| u - \PiZeroE u_I \|_{0,E}
\leq
\| u - \PiZeroE u \|_{0,E}
+
\| \PiZeroE u - \PiZeroE u_I \|_{0,E}
\leq
\| u - \PiZeroE u \|_{0,E}
+
\| u - u_I \|_{0,E}
\]
and exploiting Lemma \ref{lm:bramble} and Lemma \ref{lm:interpolation}.
\end{remark}

\begin{lemma}[Estimate of $\|e_I\|_{\cip}$]
\label{lm:epi}
Under assumptions \textbf{(A1)} and \textbf{(A2)},
the term $\|e_I\|^2_{\cip}$ can be bounded as follows
\begin{equation*}
\| e_I \|^2_{\cip,E}
\lesssim
(\lambda_E + h_E \| \bb \|^2_{[L^\infty \mathcal D((E))]^3})  \, h_E^{2\reg} \, | u |^2_{\reg+1,\mathcal D(E)}
\, .
\end{equation*}
\end{lemma}

\begin{proof}
Using the ellipticity bound \eqref{eq:ellitticità}, Lemma \ref{lm:interpolation}, and the definition of $\lambda_E$, for every $E \in \Omega_h$ we have that
\begin{equation*}
\begin{aligned}
\| \beps^\frac{1}{2} \nabla e_I \|^2_{0,E}
+ 
h_E \| \bb \cdot \nabla \PiZeroE  e_I  \|_{0,E}^2
+
\| \sigma^\frac{1}{2}  e_I \|^2_{0,E}
&\lesssim 
(\lambda_E + h_E \| \bb \|^2_{[L^\infty(E)]^3})  \, h_E^{2\reg} \, | u |^2_{\reg+1,E}
 \, ,
\end{aligned}
\end{equation*}
Similarly, on the boundary faces, it holds that
\[
\| \xi_{\beps, F} \PiZeroF e_I \|^2_{0,F} + \| \xi_{\bb, F} \PiZeroF e_I \|^2_{0,F}
\lesssim
\lambda_E \, h_E^{2\reg} \, |u |^2_{\reg+1,E}
+
\| \bb \|^2_{[L^\infty(E)]^3} \, h_E^{2\reg + 1} \, |u|_{\reg+1,E} \, .
\]
It remains to control the jump operator. We have 
\begin{equation*}
\begin{aligned}
J_h^E(e_I,e_I)
& =
\dfrac{1}{2} \sum_{F \subset \partial E \setminus \Gamma} \int_F \, \gamma_F \, h_F^2 \, \jump{\nabla \PiZero e_I}^2 \, {\rm d}F
+
\gamma_E \, |\partial E| \, S_h^E \, ( (I - \PiZeroE) e_I, (I - \PiZeroE) e_I) \\
& \lesssim
 h_E \, \| \bb \|^2_{[L^\infty\mathcal (D(E))]^3} \, \| \nabla \PiZero e_I \|_{0,\mathcal{D}(E)}^2 
+
 h_E \, \| \bb \|^2_{[L^\infty(E)]^3} \, \vert (I - \PiZeroE) e_I \vert^2_{1,E} \\
& \lesssim
 h_E \, \| \bb \|^2_{[L^\infty\mathcal (D(E))]^3} \vert e_I \vert^2_{1,\mathcal{D}(E)} \lesssim  \, h_E^{2\reg+1} \, \| \bb \|^2_{[L^\infty\mathcal D((E))]^3} \, \vert u \vert_{\reg+1,\mathcal{D}(E)}^2 \, . 
\end{aligned}
\end{equation*}
\end{proof}

\begin{lemma}[Estimate of $\eta_F^E$]
\label{lm:errF}
Under the mesh assumption \textbf{(A1)} and the regularity assumption \textbf{(A2)},
the term $\eta_F^E$ can be bounded as follows
\begin{equation*}
\eta_\mathcal{F}^E 
\lesssim
 \nu_E^\frac{1}{2} \, h^{\reg+2}_E \, \vert f \vert_{\reg+1,E} \, .
\end{equation*}	
\end{lemma}

\begin{proof}
Using the orthogonality of the projection $\PiZeroE$, Cauchy-Schwarz inequality, Poincaré inequality and  Lemma~\ref{lm:bramble}, we obtain
\begin{equation*}
\begin{aligned}
\eta_{\mathcal{F}}^E 
& =
\bigl(f, v_h - \PiZeroE v_h \bigr)_{0, E}
=
\bigl( (I-\PiZeroE) f, (I - \PiZeroE) v_h \bigr)_{0, E} \\
& \leq
\| (I-\PiZeroE) f \|_{0,E} \, \| (I - \PiZeroE) v_h \|_{0, E} \\
& \lesssim
\| (I-\PiZeroE) f \|_{0,E} \, \left( \min\{\sigma^{-\frac{1}{2}}, \, h_E \, \beps_0^{-\frac{1}{2}} \} \| v_h \|_{\cip, E} \right) \\
&\lesssim
\nu_E^{\frac{1}{2}} \, h_E^{\reg+2} \, \vert f \vert_{\reg+1,E} \, \| v_h \|_{\cip,E} \, .  
\end{aligned}
\end{equation*}
\end{proof}

\begin{lemma}[Estimate of $\eta_a^E$]
\label{lm:erra}
Under the mesh assumption \textbf{(A1)} and the regularity assumption \textbf{(A2)},
the term $\eta_a^E$ can be bounded as follows
\begin{equation*}
\eta_a^E 
\lesssim
 \bigl(\lambda_E^{\frac12}+ \beps_0^{-\frac12} \, \| \beps \|_{s,E} \, \bigr)\,  h_E^s \, \| u \|_{s+1,E} \, .
\end{equation*}	
\end{lemma}

\begin{proof}
By definition, we observe that:
\begin{equation*}
\begin{aligned}
\eta_a^E &= \int_E \beps \nabla u \cdot \nabla v_h \, \text{d}E - \int_E \beps \PiVecMinusOneE \nabla u_I \cdot \PiVecMinusOneE \nabla v_h \, \text{d}E - h_E \bar{\beps}_E S^E_h\bigl( (I - \PiZeroE) u_I, (I-\PiZeroE) v_h \bigr) .
\end{aligned}
\end{equation*}
Focusing on the consistency part, by adding and subtracting suitable terms and exploiting the $L^2$-orthogonality of $\PiVecMinusOneE$, Lemma \ref{lm:bramble} and Lemma \ref{lm:interpolation}, we obtain
\begin{equation*}
\begin{aligned}
&\int_E \beps \nabla u \cdot \nabla v_h \, \text{d}E - \int_E \beps \PiVecMinusOneE \nabla u_I \cdot \PiVecMinusOneE \nabla v_h \, \text{d}E \\
&\quad = \int_E \beps \, (\nabla u - \PiVecMinusOneE \nabla u) \cdot \nabla v_h \, \text{d}E + \int_E \beps \, \PiVecMinusOneE \nabla u \cdot \nabla v_h \, \text{d}E - \int_E \beps\,  \PiVecMinusOneE \nabla u_I \cdot \PiVecMinusOneE \nabla v_h \, \text{d}E \\
&\quad = \int_E \beps \, (\nabla u - \PiVecMinusOneE \nabla u) \cdot \nabla v_h \, \text{d}E + \int_E \beps \, \PiVecMinusOneE \nabla u \cdot (I - \PiVecMinusOneE) \nabla v_h \, \text{d}E \\
& \qquad + \int_E \beps \, \PiVecMinusOneE \nabla (u - u_I) \cdot \PiVecMinusOneE \nabla v_h \, \text{d}E \\
&\quad = \int_E \beps \, (\nabla u - \PiVecMinusOneE \nabla u) \cdot \nabla v_h \, \text{d}E + \int_E \beps \, \PiVecMinusOneE \nabla (u - u_I) \cdot \PiVecMinusOneE \nabla v_h \, \text{d}E \\
&\qquad + \int_E \left(\beps \, \PiVecMinusOneE \nabla u - \PiVecMinusOneE (\beps \, \PiVecMinusOneE \nabla u)\right) \cdot (I - \PiVecMinusOneE) \nabla v_h \, \text{d}E \\
&\quad \lesssim \bigl(\lambda_E^{\frac12} \, |u|_{s+1,E} + \beps_0^{-\frac12}  \, | \beps \nabla u |_{s,E}\bigr)\, h_E^s \,  \| v_h \|_{\cip,E} 
\lesssim \bigl(\lambda_E^{\frac12}  + \beps_0^{-\frac12} \, \| \beps \|_{s,E} \, \| u \|_{s+1,E} \bigr) \, h_E^s \, \| v_h \|_{\cip,E} \, .
\end{aligned}
\end{equation*}
Concerning the stabilization term, using the properties of $S_h^E$, Lemma \ref{lm:bramble}, and Lemma \ref{lm:interpolation}, we have
\begin{equation} \label{eq:stab-est}
\begin{aligned}
h_E \, \bar{\beps}_E \, S^E_h\bigl( (I - \PiZeroE) u_I, (I-\PiZeroE) v_h \bigr) 
&\lesssim 
\bar{\beps}_E \, | (I - \PiZeroE) u_I |_{1,E} |(I - \PiZeroE) v_h|_{1,E} \\
&\lesssim \lambda_E^{\frac12} \, h_E^s \, |u|_{s+1,E}\,  \| v_h \|_{\cip,E} \, .
\end{aligned}
\end{equation}
The thesis follows by combining the above estimates.
\end{proof}
\begin{lemma}[Estimate of $\eta_b^E$]
\label{lm:errb}
Under the mesh assumptions \textbf{(A1)} and the regularity assumption \textbf{(A2)},
the term $\eta_b^E$ can be bounded as follows 
\begin{equation*}
\eta_b^E
\lesssim 
 h^{\reg + \frac{1}{2}} \| u \|_{\reg+1}
 +
 \nu_E \, h^{\reg+1}_E \,
\| \bb \|_{[W^{\reg,\infty}(E)]^3} \, \| u \|_{\reg+1,E}
+ \int_{\partial E} (\bb \cdot \nn^E) \, e_I \, v_h \, {\rm d}s
 \, .
\end{equation*}			
\end{lemma}

\begin{proof}
Recalling the definition of $\bskewEh(\cdot,\cdot)$, we have to estimate two terms
\begin{equation*}
\begin{aligned}
\eta_{b,A}
& \coloneqq 
\bigl( \bb \cdot \nabla u, v_h \bigr)_{0,E}
- 
\bigl( \bb \cdot \nabla \PiZeroE u_I, \PiZeroE v_h \bigr)_{0,E} 
- 
\int_{\partial E} (\bb \cdot \nn ^E) \, (\PiZeroF - \PiZeroE) \, u_I \, \PiZeroF v_h \, {\rm d}s \, , \\
\eta_{b,B} & \coloneqq 
\bigl( \PiZeroE u_I, \bb \cdot \nabla \PiZeroE v_h \bigr)_{0, E}
- 
\bigl( u, \bb \cdot \nabla v_h \bigr)_{0,E} 
+ 
\int_{\partial E} (\bb \cdot \nn ^E) \, (\PiZeroF - \PiZeroE) \, v_h \, \PiZeroF u_I \, {\rm d}s \,  .
\end{aligned}
\end{equation*}
Following \cite{VEM-CIP-Conf}, these two terms can be split as 
\begin{equation*}
\begin{aligned}
\eta_{b,A}^E 
& =
\bigl( (I - \PiZeroE) \bb \cdot \nabla u, (I - \PiZeroE) v_h \bigr)_{0,E}  
+ 
\bigl( \PiZeroE u_I - u, \bb \cdot \nabla\PiZeroE v_h \bigr)_{0,E} \\
& \quad +
\int_{\partial E} (\bb \cdot \nn^E) \, (u - \PiZeroF u_I) \, \PiZeroE v_h \, {\rm d}s \\
& \eqqcolon  
\eta_{b,1}^E + \eta_{b,2}^E + \eta_{b,3}^E \, ,
\end{aligned}
\end{equation*}
and 
\begin{equation*}
\begin{aligned}
\eta_{b,B}^E 
&  =
\bigl(\PiZeroE u_I - u, \bb \cdot \nabla \PiZeroE v_h\bigr)_{0,E} 
+
\bigl((I - \PiZeroE)\bb \cdot \nabla u,(I - \PiZeroE)v_h\bigr)_{0,E}  \\
& \quad +
\int_{\partial E} (\bb \cdot \nn^E) \, (\PiZeroF - \PiZeroE) \, v_h \, (\PiZeroE u_I - u) \, {\rm d} s \\
& \eqqcolon 
\eta_{b,2}^E + \eta_{b,1}^E + \eta_{b,4}^E \, .
\end{aligned}
\end{equation*}
Our goal is to estimate $\eta_{b,i}^E$ for $i=1,2,3,4$. On the first one, we have that
\begin{equation*}
\begin{aligned}
\eta_{b,1}^E 
& =
\bigl( (I - \PiZeroE) \bb \cdot \nabla u, (I - \PiZeroE) v_h \bigr)_{0,E} 
\leq
\nu_E^\frac{1}{2}  \, h_E \, \| (I - \PiZeroE) \bb \cdot \nabla u \|_{0,E} \, \| v_h \|_{\cip,E} \\
&\lesssim
\nu_E^\frac{1}{2} \, h^{\reg+1}_E \vert \bb \cdot \nabla u \vert_{\reg,E}\| v_h \|_{\cip,E}  
\lesssim
\nu_E^\frac{1}{2} \, h^{\reg+1}_E \,
\| \beta \|_{[W^{\reg,\infty}(E)]^3} \, \| u \|_{\reg+1,E} \, \| v_h \|_{\cip,E} .
\end{aligned}
\end{equation*}
On the second one
\begin{equation*}
\begin{aligned}
\eta_{b,2}^E    \leq
\| \PiZeroE u_I - u \|_{0,E} \,  \| \bb \cdot \nabla \PiZeroE v_h \|_{0,E} 
\lesssim h^{\reg + \frac{1}{2}} \| u \|_{\reg+1} \| v_h \|_{\cip,E} \, .  
\end{aligned}
\end{equation*}
Finally, the two boundary terms can be estimated as

\begin{equation*}
\begin{aligned}
\eta^E_{b,3} + \eta^E_{b,4}
& = \int_{\partial E} (\bb \cdot \nn^E) \, (u - \PiZeroF u_I) \, \PiZeroE v_h \, {\rm d} s \\
& \quad + \int_{\partial E} (\bb \cdot \nn^E) \, (\PiZeroF - \PiZeroE) \, v_h \, (\PiZeroE u_I - u) \, {\rm d} s \\
& = \int_{\partial E} (\bb \cdot \nn^E) \, (\PiZeroE - \PiZeroF) \, v_h \, (u - \PiZeroF u_I + u - \PiZeroE u_I) \, {\rm d} s \\
& \quad + \int_{\partial E} (\bb \cdot \nn^E) \, (u - \PiZeroF u_I) \, \PiZeroF v_h \, {\rm d} s \\
& \lesssim \| \bb \|_{[L^\infty(E)]^3} \, \bigl( \| u - \PiZeroF u_I \|_{0,\partial E} + \| u - \PiZeroE u_I \|_{0,\partial E} \bigr) \, \| (\PiZeroF - \PiZeroE)v_h\|_{L^2(\partial E)} \\
& \quad + \int_{\partial E} (\bb \cdot \nn^E) \, e_I \, v_h \, {\rm d} s \\
& \lesssim \bigl( h_E^{\reg + \frac{1}{2}} \| \bb \|_{[L^\infty(E)]^3} | u |_{\reg+1,E} \bigr) \bigl( h_E^{-\frac{1}{2}} \min\{\sigma_0^{-\frac{1}{2}}, h_E^{-1} \beps_0^{-\frac{1}{2}}\} \| v_h \|_{\cip,E} \bigr) \\
& \quad + \int_{\partial E} (\bb \cdot \nn^E) \, e_I \, v_h \, {\rm d}s \\
& \lesssim h_E^{\reg+1}\, \nu_E^\frac{1}{2} \, \| \bb \|_{[L^\infty(E)]^3} \, | u |_{\reg+1,E} + \int_{\partial E} (\bb \cdot \nn^E) \, e_I \, v_h \, {\rm d}s \,.
\end{aligned}
\end{equation*}

\end{proof}

\begin{lemma}[Estimate of $\eta_c^E$]
\label{lm:errc}
Under the mesh assumptions \textbf{(A1)} and the regularity assumption \textbf{(A2)},
the term $\eta_c^E$ can be bounded as follows 
\begin{equation*}
\eta_c^E
\lesssim 
\bigl(\lambda_E^{\frac12} + \sigma_0^{-\frac12} h_E  \| \sigma \|_{s+1,E}  \bigr) h_E^\reg \, \| u \|_{\reg +1,E}  \, .
\end{equation*}			
\end{lemma}
\begin{proof} 
Similarly to Lemma \ref{lm:erra},  by adding and subtracting suitable terms and exploiting the $L^2$-orthogonality of $\PiZeroE$, Lemma \ref{lm:bramble} and Lemma \ref{lm:interpolation}, we obtain
\begin{equation*}
\begin{aligned}
\int_E \sigma \, u \, v_h \, \text{d}E - \int_E \sigma \PiZeroE u_I \PiZeroE v_h \, \text{d}E 
& = \int_E \sigma (u - \PiZeroE u) v_h \, \text{d}E + \int_E \sigma \PiZeroE u \, v_h \, \text{d}E \\
& \quad - \int_E \sigma \PiZeroE u_I \PiZeroE v_h \, \text{d}E \\
& = \int_E \sigma (u - \PiZeroE u) v_h \, \text{d}E + \int_E \sigma \PiZeroE u (I - \PiZeroE) v_h \, \text{d}E \\
& \quad + \int_E \sigma \PiZeroE (u - u_I) \PiZeroE v_h \, \text{d}E \\
& = \int_E \sigma (u - \PiZeroE u) v_h \, \text{d}E + \int_E \sigma \PiZeroE (u - u_I) \PiZeroE v_h \, \text{d}E \\
& \quad + \int_E \bigl(\sigma \PiZeroE u - \PiZeroE (\sigma \PiZeroE u)\bigr) (I - \PiZeroE) v_h \, \text{d}E \\
& \lesssim \bigl(\lambda_E^{\frac12} h_E^s |u|_{s+1,E} + \sigma_0^{-\frac12} h_E^s | \sigma u |_{s+1,E}\bigr) \| v_h \|_{\cip,E} \\
& \lesssim \bigl(\lambda_E^{\frac12} + \sigma_0^{-\frac12} h_E  \| \sigma \|_{s+1,E}  \bigr) h_E^\reg \, \| u \|_{\reg +1,E} \,\| v_h \|_{\cip,E} \, .
\end{aligned}
\end{equation*}
The stability term of $c_h^E(u_I, v_h)$ can be estimated similarly to \eqref{eq:stab-est}.
\end{proof}

\begin{lemma}[Estimate of $\eta_J^E$]
\label{lm:errJ}
Under the mesh assumptions \textbf{(A1)} and the regularity assumption \textbf{(A2)},
the term $\eta_J^E$ can be bounded as follows 
\begin{equation*}
\eta_J^E
\lesssim 
h_E^{\reg+\frac{1}{2}} \| \bb \|_{[L^\infty(\mathcal{D}(E))]^3} | u |_{\reg+1,\mathcal{D}(E)} \, .
\end{equation*}			
\end{lemma}
\begin{proof}
Using Cauchy-Schwarz inequality, we immediately have
\begin{equation}
\begin{aligned}
J^E_h(u_I, v_h) 
\leq 
J^E_h(u_I, u_I)^{\frac12} \, J^E_h(v_h, v_h)^{\frac12} \leq
J^E_h(u_I, u_I)^{\frac12} \, \| v_h \|_{\cip,E} \, .
\end{aligned}
\end{equation}
Since the solution $u$ is sufficiently regular, it holds that $\jump{\nabla u} = 0$ for every internal face $F$.
Hence, concerning the jump part of $J^E_h(u_I, u_I)$, it holds 
\begin{equation}
\begin{aligned}
\frac{1}{2}
\sum_{F \subset \partial E \setminus \Gamma_E} \int_F \gamma_F \, h_F^2 \, \jump{\nabla\PiZero u_I}^2 {\rm d }F 
&=
\frac{1}{2}
\sum_{F \subset \partial E \setminus \Gamma_E} \int_F \gamma_F \, h_F^2 \, \jump{\nabla(u - \PiZero u_I)}^2 {\rm d }F \\
&\lesssim
h_F^2 \, \| \bb \|^2_{[L^\infty(\mathcal{D}(E))]^3} \, \bigl \| \jump{\nabla(u - \PiZero u_I)} \bigr \|^2_{0,F} \\
& \lesssim h_E^{2\reg+1} \| \bb \|^2_{[L^\infty(\mathcal{D}(E))]^3} | u |^2_{\reg+1,\mathcal{D}(E)}
\end{aligned}
\end{equation}
Regarding the VEM stabilization part, using \eqref{eq:stab-prop}, it holds that
\[
\gamma_E \, |\partial E| \, S_h\bigl( (I - \PiZeroE) u_I, (I-\PiZeroE) u_I \bigr) 
\lesssim h_E^{2\reg+1} \| \bb \|^2_{[L^\infty(E)]^3} | u |^2_{\reg+1,E}
\]
Combining the last two inequalities, we conclude the proof
\end{proof}

\begin{lemma}[Estimate of $\eta_N^E$]
\label{lm:errN}
Under the mesh assumptions \textbf{(A1)} and the regularity assumption \textbf{(A2)},
the term $\eta_N^E$ can be bounded as follows
\begin{equation*}
\eta_N^E
\lesssim 
\bigl(
\lambda_E^\frac{1}{2} 
+
\| \bb\|^\frac{1}{2}_{[L^\infty(E)]^3} \, h_E^{\frac{1}{2}}
\bigr)
h_E^{\reg } \, | u |_{\reg+1,E}
 \, .
\end{equation*}			
\end{lemma}

\begin{proof}
By definition, we have that
\begin{equation}
\begin{aligned}
\Tilde{\mathcal{N}}_h^E(u, v_h) - \mathcal{N}_h^E(u_I, v_h) 
&=
\sum_{F \subset \Gamma_E} \Bigg(
- \int_F (\beps \PiVecMinusOneE \nabla v_h \cdot \nn^E) \, (u - \PiZeroF u_I) \, {\rm d}F \\
&\qquad\qquad\quad
- \int_F (\beps \nabla u \cdot \nn^E) \, v_h \, {\rm d}F
+
\int_F (\beps \PiVecMinusOneE \nabla u_I \cdot \nn^E) \, \PiZeroF v_h \, {\rm d}F
\\
&\qquad\qquad\quad 
+ \frac{\|\beps\|_{[L^\infty(F)]^{3 \times 3}}}{\delta h_E} \int_F (u - \PiZeroF u_I) \, \PiZeroF v_h \, {\rm d}F   \\
&\qquad\qquad\quad 
+ \frac{1}{2} \int_F \vert \bb \cdot \nn^E \vert \, (u - \PiZeroF u_I) \, \PiZeroF v_h \, {\rm d}F \Bigg) \\
& \eqqcolon \eta_{N,a} + \eta_{N,b} + \eta_{N,c} + \eta_{N,d} \, .
\end{aligned}
\end{equation}
Using the definition of the CIP norm \eqref{eq:normaCIP}, arguing as in \eqref{eq:beps-estimate-boundary}, we have that
\[
\begin{aligned}
\eta_{N,a} 
&\lesssim 
\sum_{F \subset \Gamma_E} \| \beps \PiVecMinusOneE \nabla v_h \cdot \nn^E \|_{0,F} \| u - \PiZeroF u_I \|_{0,F} \\
&\lesssim 
\sum_{F \subset \Gamma_E} \lambda_E^\frac12 \, \left(\frac{\beps_1}{\beps_0}\right)^\frac{1}{2} \, \| \beps^\frac{1}{2}  \nabla v_h \|_{0,E} \| u - \PiZeroF u_I \|_{0,F} \\
&\lesssim 
\sum_{F \subset \Gamma_E} \lambda_E^\frac12 \, \left(\frac{\beps_1}{\beps_0}\right)^\frac{1}{2} h_E^{\reg} | u |_{\reg+1,E} \, . 
\end{aligned}
\]
Exploiting the orthogonality of $\PiZeroF$, we get
\[
\begin{aligned}
\eta_{N,b} 
&=
\sum_{F \subset \Gamma_E}
- \int_F \bigl((\beps \nabla u - \beps \PiVecMinusOneE \nabla u_I )\cdot \nn^E\bigr ) \, v_h \, {\rm d}F \\
&=
\sum_{F \subset \Gamma_E}
- \int_F \bigl((\beps \nabla u - \beps \PiVecMinusOneE \nabla u_I )\cdot \nn^E\bigr ) \, \PiZeroF v_h \, {\rm d}F \\
& \qquad +
\sum_{F \subset \Gamma_E}
- \int_F \bigl((\beps \nabla u - \beps \PiVecMinusOneE \nabla u_I )\cdot \nn^E\bigr ) \, (I - \PiZeroF) v_h \, {\rm d}F
\\
& \lesssim
\sum_{F \subset \Gamma_E} \bigl(\beps_1^\frac{1}{2} h_E^\frac{1}{2} \| \nabla u - \PiVecMinusOneE \nabla u_I \|_{0,F}\bigr) \, \bigl( \Vert \xi_{\beps, F} \PiZeroF v_h \Vert_{0,F} + \beps_1^\frac{1}{2}h_E^{-\frac{1}{2}} \|  (I - \PiZeroF) v_h\|_{0,F}\bigr) \\
& \lesssim
\sum_{F \subset \Gamma_E} \bigl(\beps_1^\frac{1}{2} h_E^{\reg} | u |_{\reg+1,E}\bigr) \, \left( \Vert \xi_{\beps, F} \PiZeroF v_h \Vert_{0,F} + \left(\frac{\beps_1}{\beps_0}\right)^\frac{1}{2} \|  \beps^\frac{1}{2}  \nabla v_h\|_{0,E}\right) \\
&\lesssim
\sum_{F \subset \Gamma_E} \lambda_E^\frac{1}{2} h_E^{\reg} | u |_{\reg+1,E} \| v_h \|_{\cip,E}\, 
\end{aligned}
\]
On the third one, we have
\[
\begin{aligned}
\eta_{N,c} 
\lesssim
\sum_{F \subset \Gamma_E}
\beps_1^\frac{1}{2} h_E^{-\frac{1}{2}} \| u - \PiZeroF u_I \|_{0,F} \| \xi_{\beps, F} \PiZeroF v_h \| \lesssim \lambda_E^\frac{1}{2} h_E^{\reg } | u |_{\reg+1,E} \| v_h \|_{\cip,E} \, .
\end{aligned}
\]
Similarly, on the last one, we have that
\[
\begin{aligned}
\eta_{N,d} 
\lesssim
\sum_{F \subset \Gamma_E}
\| \bb\|^\frac{1}{2}_{[L^\infty(F)]^3} \| u - \PiZeroF u_I \|_{0,F} \| \xi_{\bb, F} \PiZeroF v_h \| \lesssim \| \bb\|^\frac{1}{2}_{[L^\infty(E)]^3}h_E^{\reg + \frac{1}{2}} | u |_{\reg+1,E} \| v_h \|_{\cip,E} \, .
\end{aligned}
\]
\end{proof}
\begin{theorem}
Under the mesh assumptions \textbf{(A1)} and the regularity assumption \textbf{(A2)}, let $u$ be the solution of \eqref{eq:problema-continuo} and let $u_h$ be the solution of the discrete problem \eqref{eq:problema-discreto}. 
Then it holds
\begin{equation}
\label{eq:final_error_squared}
\begin{split}
\| u - u_h \|^2_{\cip} 
\lesssim \sum_{E \in \Omega_h} \Theta_E \Bigl( & \lambda_E 
+ \beps_0^{-1} \| \beps \|^2_{s,E} 
+ \sigma_0^{-1} h_E^2 \| \sigma \|^2_{s+1,E} \\
& + h_E \bigl( 1 + \| \bb \|_{[L^\infty (E)]^3} + \| \bb \|^2_{[L^\infty (\mathcal{D}(E))]^3} \bigr) \\
& + \nu_E h_E^2 \| \bb \|^2_{[W^{s,\infty}(E)]^3} + \nu_E h_E^4 \Bigr) \, ,
\end{split}
\end{equation}
where the constant $\Theta_E$ depends on $\| u \|_{\reg+1}$ and $\| f \|_{\reg}$.
\end{theorem}
\begin{proof}
The proof is obtained by combaining Proposition \ref{prp:abstract} with the following Lemmas. We note that since the functions in $V_h^k(\Omega_h)$ are continuous, it holds
\[
\sum_{E \in \Omega_h} \int_{\partial E \setminus \Gamma} (\bb \cdot \nn^E) \, e_I \, v_h \, {\rm d}s = 0
\]
It remains to control the boundary faces. These can be estimated as
\[
\int_{\Gamma_E} (\bb \cdot \nn^E) \, e_I \, v_h \, {\rm d}s
\lesssim
\| \bb\|^\frac{1}{2}_{[L^\infty(E)]^3}h_E^{\reg + \frac{1}{2}} | u |_{\reg+1,E} \| v_h \|_{\cip,E} \, .
\]
\end{proof}

\section{Numerical results}\label{sec:num}

In this section, we present a series of three-dimensional numerical experiments to validate the theoretical convergence rates and assess the actual robustness of the proposed CIP-stabilized Virtual Element Method. We mention that the code has been implemented within the C++ library \texttt{VEM++} \cite{vem++}, a platform specifically tailored for the development and testing of Virtual Element Methods.

\paragraph{Convergence results}
In this numerical test, we verify that the proposed scheme achieves optimal convergence rates under the advection-dominated regime. Following the standard practice in the Virtual Element framework, we evaluate the accuracy of the method by computing the following error quantities:
\[
e_{L^2} := \sqrt{\sum_{E\in\Omega_h}\left\|u-\Pi^{0,E}_k u\right\|^2_{0,E}}\,.
\qquad
e_{H^1} := \sqrt{\sum_{E\in\Omega_h}\left\|\nabla (u-\Pi_k^{\nabla,E} u_h)\right\|^2_{0,E}}.
\]
To this end, we choose the analytical solution of the continuous problem to be the smooth function
\[
u(x,y,z) = \sin(\pi y) \sin(\pi z) e^x\,,
\]
defined on the unit cube domain $\Omega = (0,1)^3$. The parameters of the partial differential equation are set as follows:
\[
\boldsymbol{\varepsilon} = \nu
\begin{pmatrix}
2 + x & xy & xz \\
xy & 2 + y & yz \\
xz & yz & 2 + z
\end{pmatrix} , \qquad
\mathbf{b} = \begin{pmatrix}
y + z \\
x + z \\
x - y
\end{pmatrix} , \qquad 
\sigma = 1 + x^2 + y^2 + z^2 \,.
\]
To simulate a strongly advection-dominated regime, we select the diffusion scaling parameter as $\nu = 10^{-6}$. Additionally, the Nitsche boundary parameter is fixed to $\delta = 0.1$, and the CIP stabilization parameter is set to $\gamma = 0.025$. 

For our numerical assessment, we consider three distinct families of polyhedral meshes on the unit cube: the first family consists of structured  cubes and octahedra; the second family is made of random Voronoi polyhedra; the third family comprises regularized Voronoi polyhedra.
These grids are depicted in Figure~\ref{fig:meshes}. 

The corresponding error curves are plotted in Figure~\ref{fig:test1-1}. As expected, the plots clearly show that the method achieves the optimal convergence rates in all scenarios for different polynomial degrees $k$, confirming the robust behavior of the formulation.

\begin{figure}[htbp]
\centering
\begin{tabular}{ccc}
\includegraphics[width=0.30\textwidth]{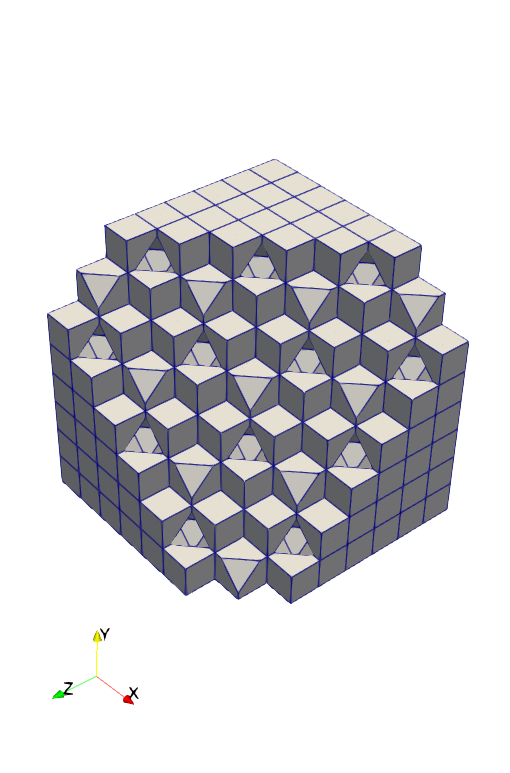}  &
\includegraphics[width=0.30\textwidth]{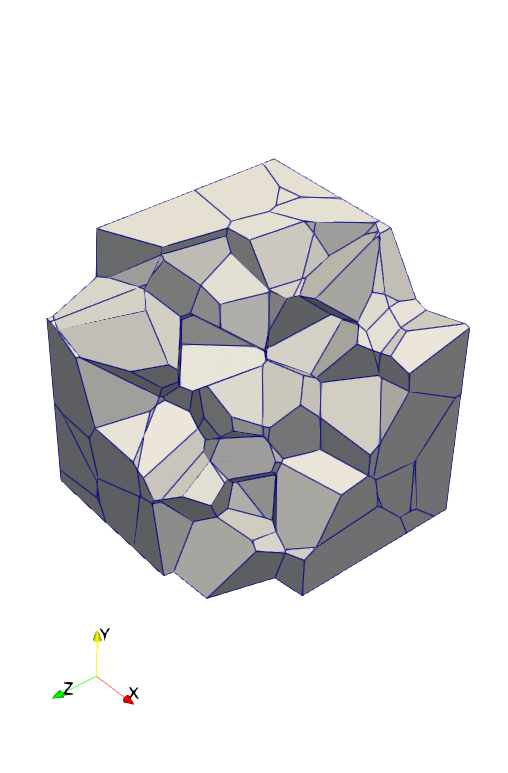} &
\includegraphics[width=0.30\textwidth]{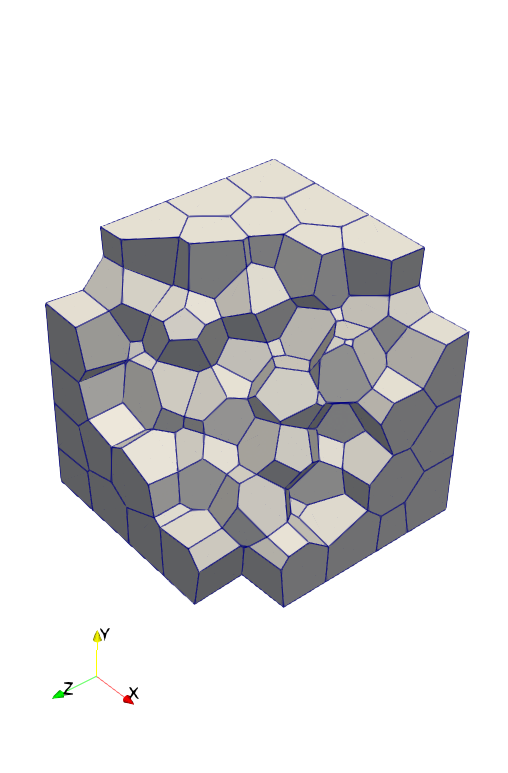}
\end{tabular}
\caption{Different types of meshes considered in the first experiment.}
\label{fig:meshes}
\end{figure}

\begin{figure}[htbp]
\centering
\begin{tabular}{cc}
\includegraphics[width=0.42\textwidth]{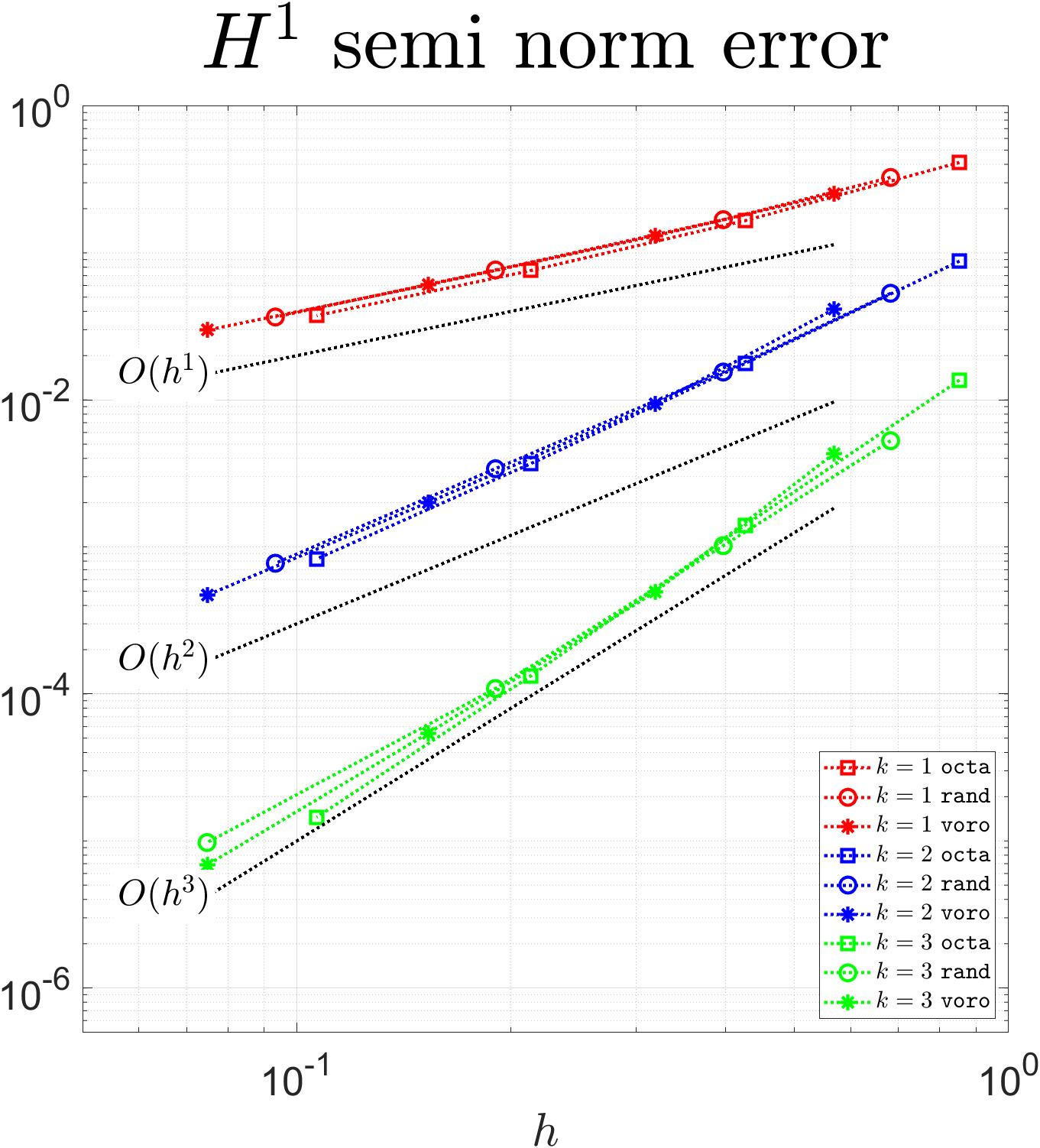}  &
\includegraphics[width=0.42\textwidth]{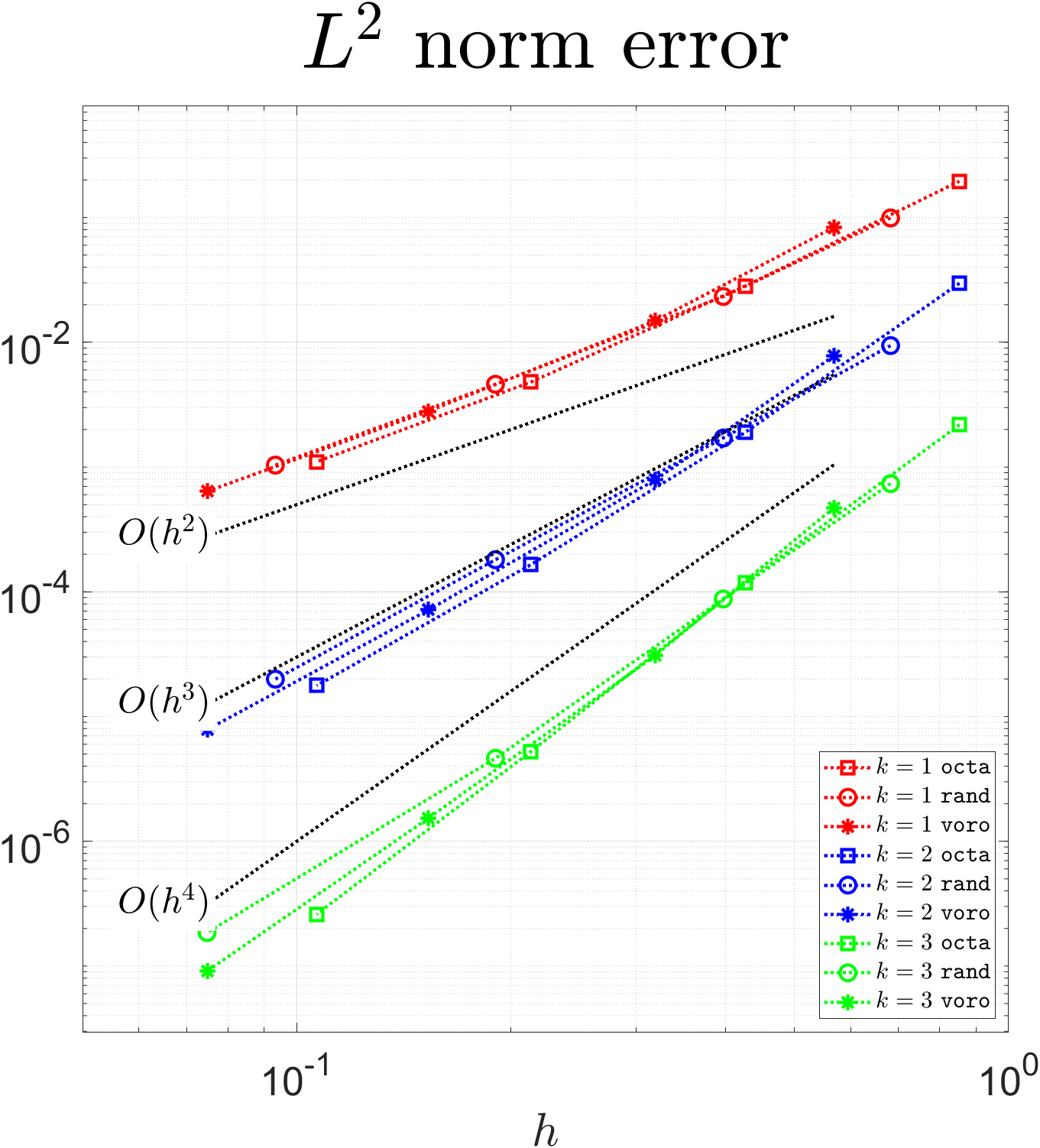}
\end{tabular}
\caption{Convergence results on different types of meshes and different values of $k$.}
\label{fig:test1-1}
\end{figure}

\paragraph{Solution with a boundary layer}
In this numerical experiment, we consider an isotropic diffusion tensor given by $\nu \mathbf{I}$, where $\nu > 0$ acts as a scaling factor. The exact analytical solution to \eqref{eq:problema-continuo} is prescribed as:
\[
u (x,y,z) = \frac{1 - x + x \exp( - \nu^{-1}) - \exp(-x \nu^{-1})}{1 - \exp(-\nu^{-1})} \, .
\]
This function exhibits a sharp boundary layer at $x=0$, while behaving approximately as $1-x$ throughout the remainder of the domain. Consequently, mesh refinement is superfluous in regions far from the layer, whereas a higher density of elements is strictly required near $x=0$ to properly capture the steep solution gradient.

To this end, we construct a locally refined mesh over the unit cube $\Omega = [0,1]^3$, originating from an initial uniform grid composed of 64 cubes. At each subsequent refinement iteration, the elements within the layer immediately adjacent to the boundary $x=0$ are subdivided by halving their edge lengths. Crucially, this localized refinement strategy introduces hanging nodes into the grid. In this way, we construct a squence of meshes from \texttt{mesh0} to \texttt{mesh4}. A cubic mesh structure is selected to obtain a good visualization of the numerical solution; furthermore, thanks to the inherent flexibility of the VEM framework, elements featuring hanging nodes are naturally treated as general polytopes, completely bypassing the need for complex, conforming transition algorithms. For this test, the problem parameters are set to $\nu = 10^{-4}$, $\sigma \equiv 0$, and $\bb = [-1, 0, 0]^T$, while the CIP penalty parameter is chosen as $\gamma = 0.25$, and $\delta = 1/500.$

Figure \ref{fig:test2-1} reports a comparison between the numerical solutions obtained with and without the CIP stabilization. The standard, unstabilized formulation exhibits severe spurious oscillations that propagate upstream from the boundary layer, culminating in a prominent unphysical peak of approximately $2.4$. Conversely, the CIP method successfully suppresses these global instabilities, delivering a remarkably smooth profile and restricting the localized overshoot near the layer to a maximum value of $1.3$.

In Figure \ref{fig:test2-2}, we track the evolution of the CIP-stabilized solution across the sequence of progressively refined meshes, from \texttt{mesh0} to \texttt{mesh4}. Specifically, the numerical solution is evaluated along a 1D line cut spanning from $(0, 0.5, 0.5)$ to $(1, 0, 0)$. As the localized mesh refinement concentrates elements near the outflow boundary $x = 0$, the numerical solution converges toward the prescribed boundary value $u = 0$. 
Concurrently, a slight increase in the magnitude of the localized overshoot is observed. While tuning the CIP parameter $\gamma$ to a higher value could further damp this residual peak, it introduces a well-known numerical trade-off; an excessive stabilization could compromise the sharp profile and causing the numerical solution to deviate further from zero at $x = 0$.

\begin{figure}[htbp]
\centering
\begin{tabular}{cc}
\includegraphics[width=0.42\textwidth]{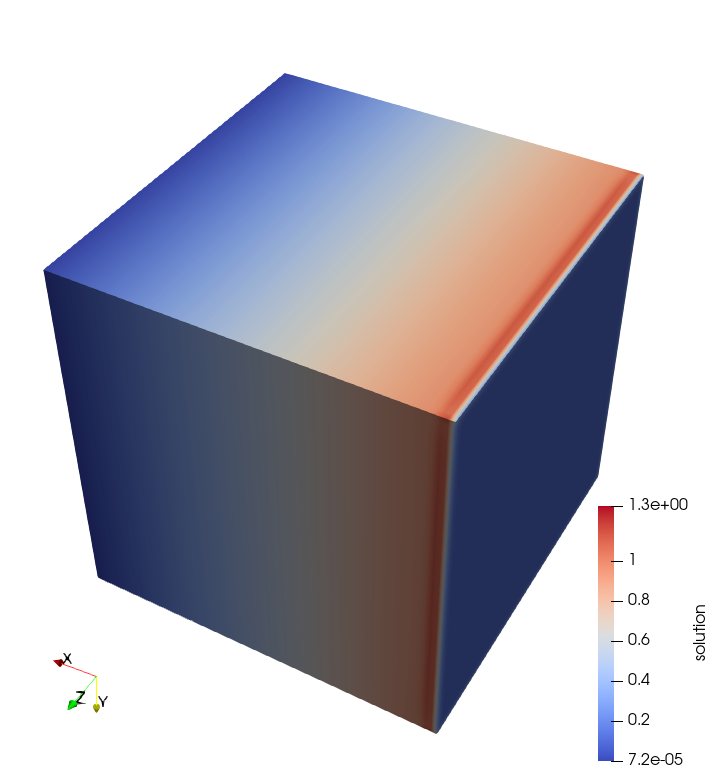}  &
\includegraphics[width=0.42\textwidth]{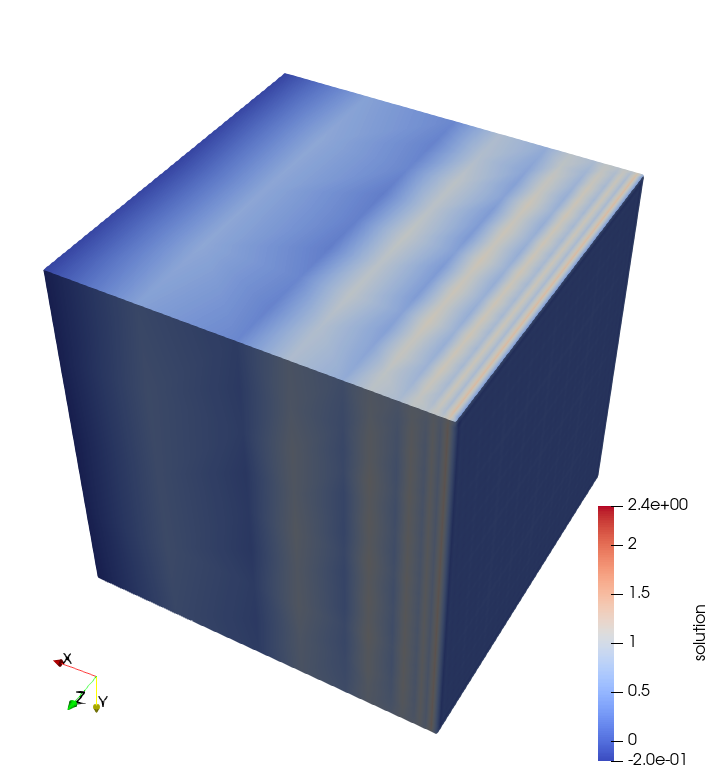}
\end{tabular}
\caption{Numerical solutions obtained without CIP (left) and with CIP (right).}
\label{fig:test2-1}
\end{figure}

\begin{figure}[htbp]
\centering
\includegraphics[width=0.42\textwidth]{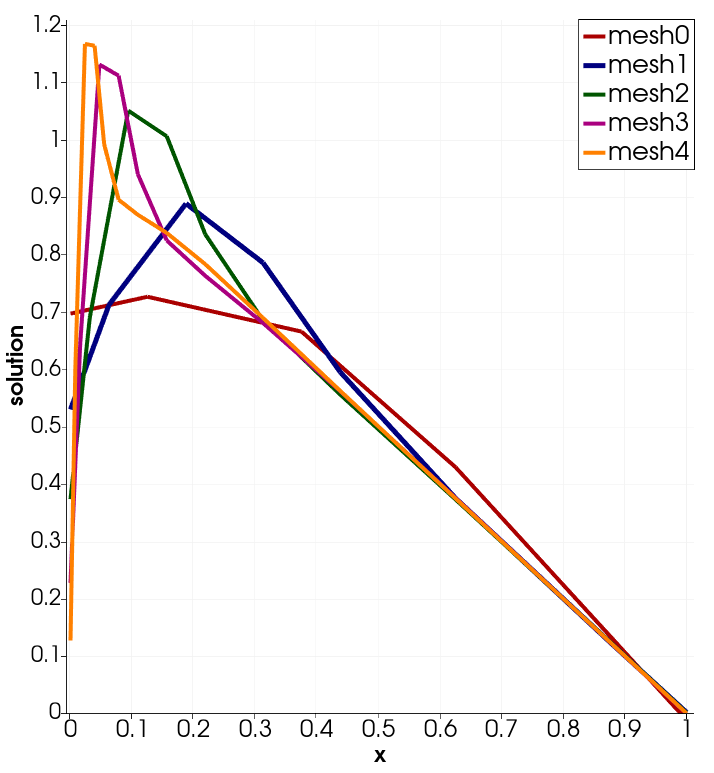}  
\caption{Numerical solutions on the line connecting (0,0.5,0.5) to (1, 0.5, 0.5) for different meshes.}
\label{fig:test2-2}
\end{figure}

\paragraph{Advection of a Discontinuous Profile}
In this numerical experiment, we investigate the transport of a discontinuous scalar field injected at the inflow boundary $z = 0$, driven by the uniform, skew advective field $\bb = [0.15, 0.25, 1.0]^T$. 
The problem is studied in the strongly convection-dominated regime by setting the isotropic diffusion coefficient to $10^{-6}$ and the source term to $f = 0$. 
Furthermore, the reaction coefficient is identically zero; as discussed in \cite{VEM-CIP-Conf}, the VEM-CIP stabilization retains its robustness for divergence-free, linear advective fields even in the absence of a reaction term. 

To impose the Dirichlet boundary conditions weakly, we set the Nitsche penalty parameter to $\delta = 0.05$. 
Specifically, we apply Dirichlet boundary conditions on the inflow faces $x=0$, $y=0$, and $z=0$, while on the remaining faces we impose homogeneous Neumann boundary conditions, allowing the scalar quantity to exit the domain freely.
On the inflow boundary $z=0$, we prescribe the following discontinuous profile:
$$
g(x,y,0) = 
\begin{cases}
1 & \text{if } 0.15 \leq x \leq 0.35, \, 0.15 \leq y \leq 0.35, \\
1 & \text{if } 0.15 \leq x \leq 0.35, \, 0.65 \leq y \leq 0.85, \\
0 & \text{otherwise},
\end{cases}
$$
while on the remaining Dirichlet faces we set $g \equiv 0$. 

To highlight the effectiveness of the proposed stabilization, we compare the standard Galerkin VEM approach, where the CIP term is neglected, against the CIP-stabilized formulation using a penalty parameter $\gamma = 0.1$.
Finally, the computational domain is discretized using a uniform cubic mesh, obtained by subdividing each edge into equal segments, which ensures an accurate and clear visualization of the advected profile.

The numerical solutions obtained on a uniform cubic mesh (with each edge subdivided into 40 segments) are depicted in Figure \ref{fig:test3}. Consistent behavior is observed across different mesh resolutions. 
In the standard Galerkin framework, the solution exhibits significant spurious oscillations as it approaches the upper boundary $z=1$. Furthermore, the numerical solution displays a non-physical overshoot, reaching a maximum value of approximately $2.2$. 
In contrast, by incorporating the CIP bilinear form into the formulation, the numerical solution becomes remarkably smoother. With the proposed stabilization, the maximum value is successfully bounded, remaining close to the expected physical peak of $1$.

Figure \ref{fig:test3-2} displays the cross-sectional profiles at the slice $z = 0.5$. 
It should be noted that the reduction of the peak value from 1.0 to approximately 0.58 at z=0.5 is attributed to the combined effect of the transport-induced geometric spreading of the advected profile and the numerical dissipation introduced by the CIP stabilization. While the latter is necessary to ensure a good approximation of the solution, it inherently provides a localized smoothing of the sharp gradients.
Interestingly, for a similar two-dimensional problem investigated in \cite{trezzihyp}, it was proven that the CIP stabilization outperforms the SUPG stabilization \cite{Dauphin2026} when odd values of $k$ are employed. Conversely, the SUPG approach exhibits a superior behavior for even polynomial degrees $k$.
\begin{figure}[htbp]
\centering
\begin{tabular}{cc}
\includegraphics[width=0.42\textwidth]{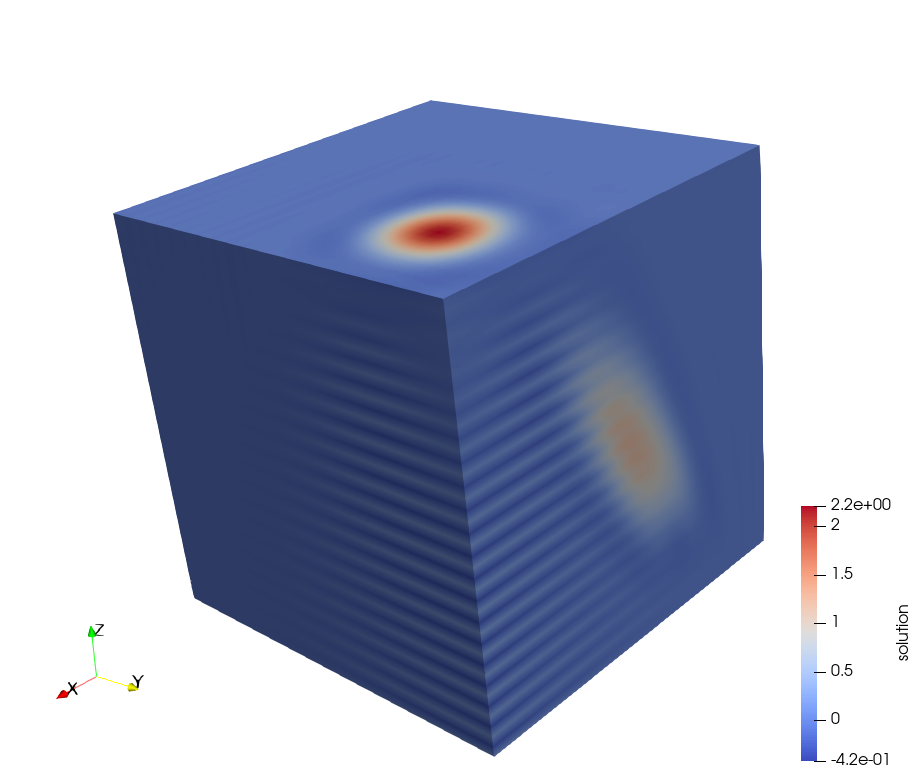}  &
\includegraphics[width=0.42\textwidth]{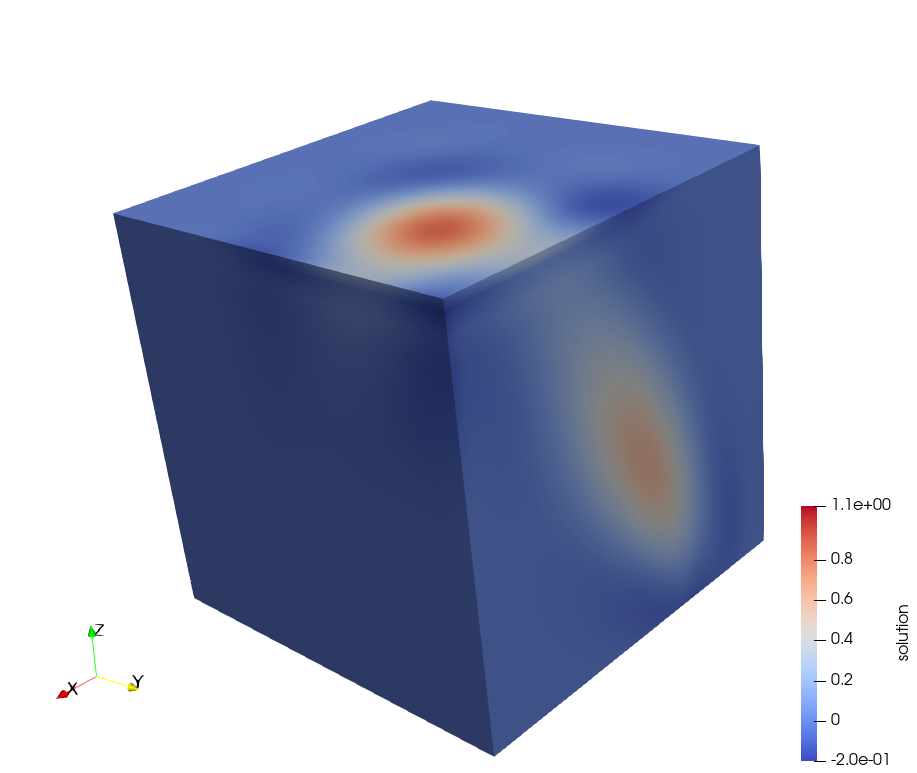}
\end{tabular}
\caption{Numerical solutions obtained without CIP (left) and with CIP (right).}
\label{fig:test3}
\end{figure}

\begin{figure}[htbp]
\centering
\begin{tabular}{cc}
\includegraphics[width=0.42\textwidth]{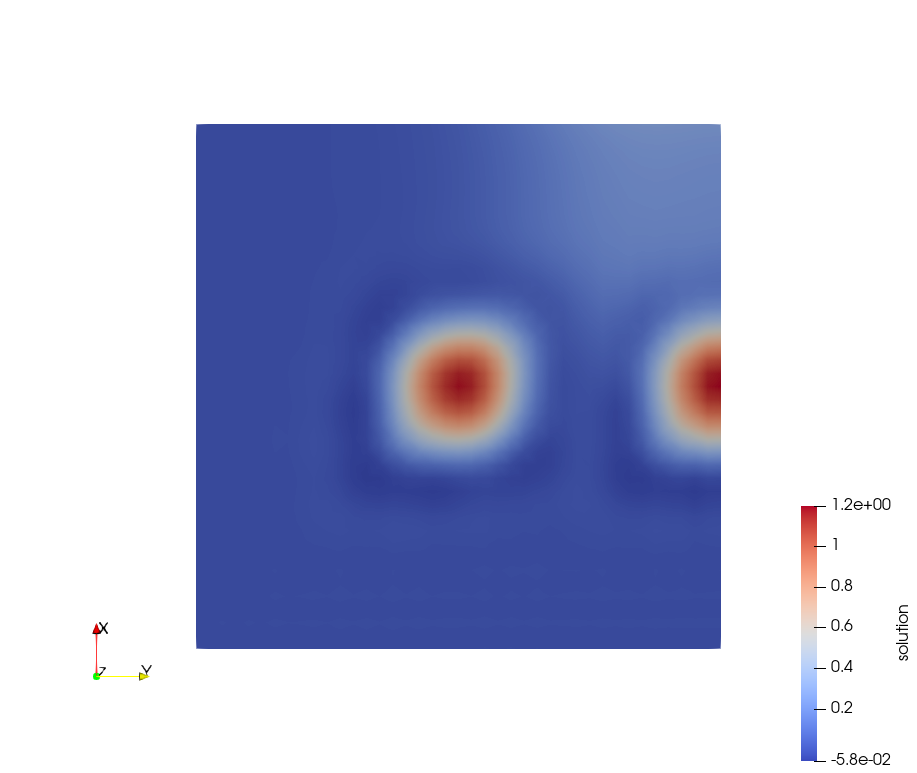}  &
\includegraphics[width=0.42\textwidth]{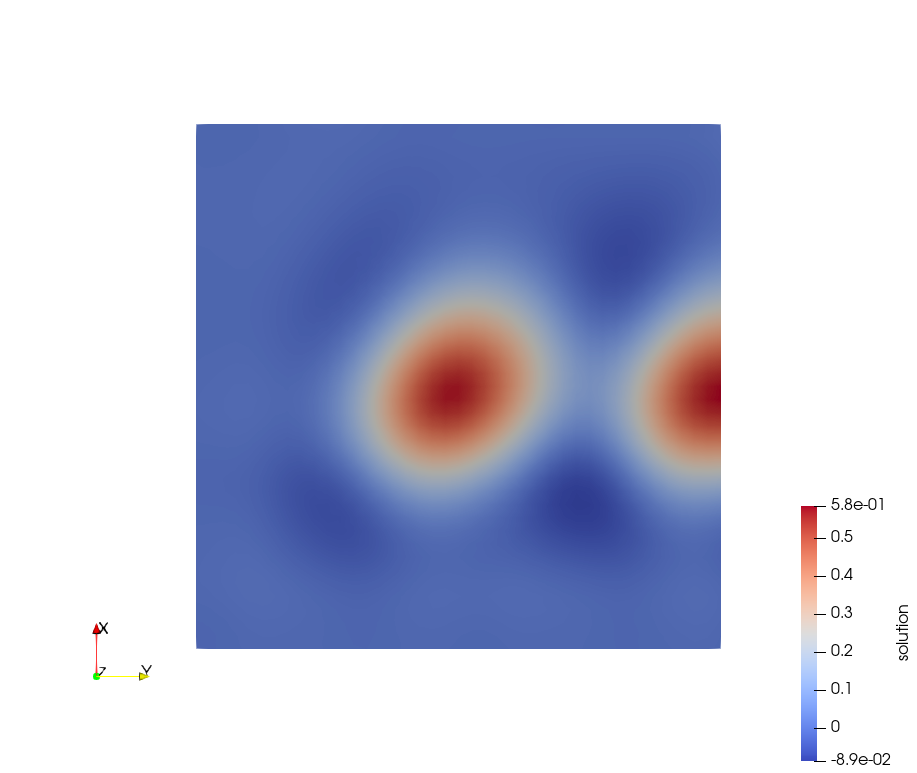}
\end{tabular}
\caption{Numerical solutions obtained without CIP (left) and with CIP (right) at the slice $z = 0.5$.}
\label{fig:test3-2}
\end{figure}

\section*{Acknowledgments}
MT has been funded by the European Union (ERC, NEMESIS, project number 101115663).
Views and opinions expressed are however those of the author only and do not necessarily reflect those of the EU or the ERC Executive Agency.
MT also acknowledges the support of the INdAM-GNCS Project, code CUP\_ E53C25002010001.

\bibliographystyle{plain}
\bibliography{biblio}

\end{document}